\documentclass[]{article}
\usepackage{graphicx}
\usepackage[utf8]{inputenc}
\usepackage[T1]{fontenc}
\usepackage{aeguill}
\usepackage{amssymb}

\usepackage{bbm}
\usepackage{amsmath}

\usepackage{fancyvrb,color}
\newtheorem{theorem}{Theorem}
\newtheorem{lemma}[theorem]{Lemma}

\newtheorem{Proposition}[theorem]{Proposition}

\newtheorem{remark}{Remark}
\newtheorem{example}{Example}

\newtheorem{definition}{Definition}

\newcommand{\Vect}[1]{
  \mathbf{#1}
}

\newcommand{\tr}[1]{
 {#1}^{\!  T}
}

\newcommand{\PrSc}[2]{
  \tr{#1} {#2}
}

\newcommand{\Mat}[1]{
  \underline{\mathbf{#1}}
}

\newcommand{\Esp}{
  \mathbf{E}
}

\newcommand{\RR}[1][2]{\mathbb{R}^{#1}}
\newcommand{\ZZ}[1][2]{\mathbb{Z}^{#1}}
\newcommand{\NN}[1][]{\mathbb{N}^{#1}}

\newcommand{\mSp}{\mathbbm{M}
}

\newcommand{\mm}{\lambda^{\mE}
}
\newcommand{\mE}{\mathbbm{m}
}

\newcommand{\cSp}[1][]{{\Omega}#1
}

\newcommand{\sSp}{\mathbbm{S}
}
\newcommand{\sB}{\mathcal{B}}

\newcommand{\sm}{\mu
}
\newcommand{\xm}[2][x]{#1^{#2}
}
\newcommand{\ism}[2][\Lambda\times\mSp]{
\int_{#1} #2 \sm(d\xm{m})
}

\newcommand{\VE}{V}

\newcommand{\Par}{\theta}                    % param�tre
\newcommand{\ParV}{\Vect{\theta}}            % param�tre vectoriel

\newcommand{\ParVT}{\Vect{\theta}^\star}     % param�tre vectoriel Vrai
\newcommand{\SpPar}{\Vect{\Theta}}           % espace des param�tres

\newcommand{\VIPar}[3]{\VE \left(#1|#2  ; #3 \right)}

\newcommand{\SExI}[3]{V_{#1}(#2|#3)}
\newcommand{\cSEx}[2]{\kappa^{(#1)}_{#2}}

\newcommand{\Det}[1][]{{\bf [Det{#1}]}~}
\newcommand{\DetT}[1][]{{\bf[$\widetilde{\mbox{\bf Det}}${#1}]}~} 
\title{\sc  Takacs-Fiksel method for stationary marked Gibbs point processes}

\author{J.-F. Coeurjolly$^{1,2}$, D. Dereudre$^{3}$, R. Drouilhet$^{2}$ and F. Lavancier$^{4}$  \vspace*{1cm}
\\
{\small $^{1}$ GIPSA-lab, Grenoble University, France,}\\
{\small  $^{2}$ Laboratoire Jean Kuntzmann, Grenoble University, France} \\
{\small $^{3}$ LAMAV UVHC FR 2956, Lille Nord de France University, France, }\\
{\small $^{4}$ Laboratoire de Mathématiques Jean Leray, Nantes University, France.}}
\begin{document}
\bibliographystyle{plain}
\maketitle

%%%%%%%%%%%%%%%%%%%%%%%%%%%%%%%%%%%

%%%part(main)
\bigskip

\begin{abstract}
This paper studies a method to estimate the parameters governing the distribution
of a stationary marked Gibbs point process. This procedure, 
known as the Takacs-Fiksel method, is based on the estimation of the left and right hand sides of the Georgii-Nguyen-Zessin formula
and leads to a family of estimators due to the possible choices of test functions. 
We propose several examples illustrating 
the interest and flexibility of this procedure. We also provide  sufficient 
conditions based on the model and the test functions to derive asymptotic 
properties (consistency and asymptotic normality) of the resulting estimator. 
The different assumptions are discussed for exponential family models and for 
a large class of test functions. A short simulation study is proposed to assess the correctness of the methodology and the asymptotic results.
\end{abstract}


%%%%%%%%%%%%%%%%%%%%%%%%%%%%%%%%%%%%%%%%%%%%%%%%%%%%%%%%%%%%%%%%%%%%%%%%%%%%%%
%%%%%%%%%%%%%%%%%%%%%%%%% Motivations %%%%%%%%%%%%%%%%%%%%%%%%%%%%%%%%%%%%%%%%
%%%%%%%%%%%%%%%%%%%%%%%%%%%%%%%%%%%%%%%%%%%%%%%%%%%%%%%%%%%%%%%%%%%%%%%%%%%%%%
\section{Introduction}

Spatial point pattern data arise in a wide range of applications and since the seventies \cite{A-Bri75}, \cite{A-Kri78} various statistical methods have been developed to study these kinds of data (see  \cite{B-MolWaa03}, \cite{MW07} or \cite{B-IllPenSto08} for recent reviews). In particular, a spatial point process is often modelled as the realization of a Gibbs distribution, defined through an interaction function, also called Hamiltonian. Gibbs models are extensions of the well-known Poisson process since they constitute a way to introduce dependence between points. Inference for parametric models in this setting has known a large development during the last decade. The most popular method to estimate the parameters is  certainly the maximum likelihood estimator (MLE). It involves an intractable normalizing constant, but recent developments in computational statistics, in particular perfect simulations, have made  inference feasible for many Gibbs models (see \cite{A-BerMol03}). Although the MLE suffers from a lack of theoretical justifications (only some results for sparse patterns are proposed in \cite{A-Mase92}), some comparison studies, as in \cite{A-Diggle94}, have shown that it outperforms the other estimation methods. Nevertheless, the computation of the MLE remains very time-consuming and even extremely difficult to perform for some models. It is thus necessary to have   quick alternative  estimators at  one's disposal, at least to propose relevant initial values for the MLE computation. The maximum pseudo-likelihood estimator (MPLE for short) constitutes one of them. Proposed by  Besag in \cite{A-Bes74} and popularized by J.L. Jensen and J. M\o{}ller in \cite{A-JenMol91} and A. Baddeley and R. Turner in \cite{A-BadTur00} for spatial point processes, this method has the advantage of being theoretically well-understood (see \cite{A-BilCoeDro08}, \cite{A-DerLav09}) and of being much faster to compute than the MLE. Another estimation procedure is the Takacs-Fiksel estimator, which arose from \cite{A-Fiksel84}, \cite{A-Takacs86},  \cite{A-Fiksel88}. It can be viewed, in some sense, as a generalization of the MPLE. As a matter of  fact, the Takacs-Fiksel method is not very popular, nor really used in practice. The main reason is certainly its relatively poor performances, in terms of mean square error, observed for some particular cases as in  \cite{A-Diggle94}.  However, we think that this procedure deserves some consideration for several reasons that we expose below.\\

The Takacs-Fiksel procedure is based on the Georgii-Nguyen-Zessin formula (GNZ formula for short). Empirical counterparts of the left and the right hand sides of this equation are considered, and the induced estimator is such that the difference of these two terms is close to zero. Since the GNZ formula is valid for any test functions, the Takacs-Fiksel procedure does not only lead to one particular estimator but to a family of estimators, depending on the choice of the test functions. This flexibility is the main advantage of the procedure. We present several examples in the present paper. Let us summarize them in order to underline the interest of this method (see Section \ref{sec-examples} for more details). First, this procedure can allow us to achieve estimations that likelihood-type methods cannot. As an example, we focus on the quermass model, which gathers the area interaction point process  as a particular case (see \cite{A-Kendall99}, \cite{A-Der09}). This model is sometimes used for geometric random objects. From a data set, one typically does not observe the point pattern but only some geometric sets arising from these points. The non-observability of the points makes the likelihood-type inference unfeasible. As we will show, this problem may be solved thanks to the Takacs-Fiksel procedure, provided that the test functions are chosen properly.

Another motivation is the possibility to choose test functions depending on the Hamiltonian in order to construct quicker estimators which  do not require the computation of an integral for each value of the parameter. This improvement appears crucial for rigid models, such as those involved in stochastic geometry (see \cite{A-DerLav10}), for which the MLE is prohibitively time-consuming and the MPLE still remains difficult to implement. Moreover, for some models, it is  even possible to obtain explicit estimators which do not require any simulation nor optimization. This is illustrated for the Strauss model.

Therefore, it appears important to us to understand the theoretical properties of this procedure. This problem is the main objective of the present paper. We prove the consistency and the asymptotic normality of the induced estimator in a very general setting. In particular, we obtain a central limit theorem for the Takacs-Fiksel estimator with the classical rate of convergence, i.e. square root of the volume of the observation domain. This asymptotic result leads to the following comment: as a quick consistent estimator, the Takacs-Fiksel estimator appears to be a very good starting point for refined algorithms. Among them, let us mention the first step Newton-Raphson algorithm (as used in \cite{A-Huang99}) which allows an accurate approximation of the MLE, starting from a consistent estimator, in only one step. Although the theoretical justifications are missing in the Gibbs framework, it is well-known that this procedure leads to an efficient estimator in the classical iid case (see \cite{B-Lehman}). Another possibility could be to exploit the local asymptotic normality of the model in order to construct an adaptive estimator from the Takacs-Fiksel estimator. The Local Asymptotic Normality  property (LAN) has only been proved for restrictive models in \cite{A-Mase92}, but one can hope that it remains true for most Gibbs models.  This procedure also leads to an efficient estimator. All these possibilities are interesting prospects for future investigations.\\

Some asymptotic properties of the Takacs-Fiksel procedure have already been investigated in two previous studies: one by L. Heinrich in \cite{Heinrich92} and the one by J.-M. Billiot in \cite{A-Billiot97}. These papers have different frameworks and are based on different tools, but they both involve regularity and integrability type assumptions on the Hamiltonian and a theoretical condition which ensures that the contrast function (associated to the Takacs-Fiksel procedure) has a unique minimum. In \cite{Heinrich92}, the consistency and the asymptotic normality are obtained for a quite large class of test functions. These results are, however, proved under the Dobrushin condition (see Theorems~2 and~3 in  \cite{Heinrich92}) which implies the uniqueness of the underlying Gibbs measure and some mixing properties. This condition imposes a dramatic reduction of the space of possible values for the parameters of the model. In \cite{A-Billiot97}, the author focuses only on pairwise interaction point processes (which excludes the quermass model for example). The author mainly obtained the consistency for a specific class of test functions. In the case of a multi-Strauss pairwise interaction point process, the author also proved that the identifiability condition holds for the class of test functions he considered. 

In contrast, our asymptotic results are proved in a very general setting, i.e. for a large class of stationary marked Gibbs models and test functions.  The method employed to prove asymptotic normality is based on a conditional centering assumption, first appeared in \cite{guyonkunsch92} for the Ising model and generalized to certain spatial point processes in \cite{A-JenKun94}. The main restriction that this method induces is only the finite range of the Hamiltonian. There are no limitations on the space of parameters and, in particular, the possible presence of phase transition does not affect the asymptotic behavior of the estimator. Moreover, the test functions may depend on the parameters. This extension seems important to us because, as emphasized in Section \ref{sec-quick},  such test functions can lead to quick and/or explicit estimators. All the general hypotheses assumed for the asymptotic results are discussed. For  this, we focus on exponential family models, that is, on models whose interaction function is linear in the parameters. We show that our integrability and regularity assumptions are not restrictive since they are valid for a large class of models such as the Multi-Strauss marked point process, the Strauss-disc type point process, the Geyer's triplet point process, the quermass model and for all test functions used as a motivation for this work. In the setting of the exponential family models, we also discuss the classical identifiability condition which is required for the Takacs-Fiksel procedure. To the best of our knowledge, this is the first attempt to discuss it. We will specially 
dwell on questions like: what choices of test functions (and how many test functions) lead to a unique minimum of the contrast function? We propose general criteria and provide examples. It seems commonly admitted that to achieve the identification of the Takacs-Fiksel procedure, one should at least choose as many test functions as the number of parameters. As a consequence of our study, it appears that one should generally strictly choose  more test functions than the number of parameters to achieve identification. 

The rest of the paper is organized as follows. Section~\ref{sec-background} introduces notation and a short background on marked Gibbs point processes. The Takacs-Fiksel method is presented in Section~\ref{TFprocedure}. It is based on the GNZ formula which is recalled in Section~\ref{TFprocedure} also.  Several examples of test functions are given. They aim at illustrating our interest in considering the Takacs-Fiksel procedure. The asymptotic results of the induced estimator are proposed in Section~\ref{sec-results}. Our results are obtained from a single realization observed in a domain whose volume is supposed to increase to infinity. Some integrability and regularity assumptions made for the Hamiltonian and for the test functions are discussed in this section while in Section~\ref{SectionIdent} the identifiability condition is specifically dealt with. In Section \ref{non-hereditary}, the very special situation where the energy function is not hereditary is considered. The GNZ formula is no longer valid in this setting, but  it has been recently extended  in \cite{A-DerLav09} thanks to a slight modification.  This leads to a natural generalization of the Takacs-Fiksel procedure. In Section~\ref{sec-simuls}, we propose a brief simulation study in order to assess the correctness of the asymptotic results we obtained in Section~\ref{sec-results}. Finally, Section~\ref{sec-proofs} contains the proofs of the asymptotic results.

%%%%%%%%%%%%%%%%%%%%%%%%%%%%%%%%%%%%%%%%%%%%%%%%%%%%%%%%%%%%%%%%%%%%%%%%%%%%%%%%%%%%%
%%%%%%%%%%%%%%%%%%%%%%%%%% Background et méthode %%%%%%%%%%%%%%%%%%%%%%%%%%%%%%%%%%%%
%%%%%%%%%%%%%%%%%%%%%%%%%%%%%%%%%%%%%%%%%%%%%%%%%%%%%%%%%%%%%%%%%%%%%%%%%%%%%%%%%%%%%

\section{Background and notation} \label{sec-background}

\subsection{General notation, configuration space}
%$\sB(\RR[d])$  $\RR[d]$ $\Lambda^c$ $\Subset$

Subregions of $\RR[d]$ will typically be denoted by $\Lambda$ or $\Delta$ and will always be assumed to be Borel with positive Lebesgue measure. We write $\Lambda \Subset \RR[d]$ if $\Lambda$ is bounded. $\Lambda^c$ denotes the complementary set of $\Lambda$ inside $\RR[d]$.
The notation $|.|$ will be used without ambiguity for different kind of objects. For a countable set $\mathcal J$, $|\mathcal J|$ represents the number of elements belonging to $\mathcal J$; For $\Lambda\Subset\RR[d]$, $|\Lambda|$ is  the volume of $\Lambda$; For $x \in \RR[d]$, $|x|$ corresponds to its uniform norm while $\| x\|$ is  its Euclidean norm.
For all $ x\in\RR[d],\rho>0$, %and $i\in\ZZ[d]$,
 let $\mathcal{B}( x,\rho):=\{  y \in \RR[d],\|  y-  x\|<\rho\}$. 
%and $\mathbbm{B}(i,\rho) := \mathcal{B}(i,\rho) \cap \ZZ[d]$.
%norm : $|.|$ $|\Vect x|$  $|\mathcal J|$  $|\Delta|$ 
For a matrix $\Mat M$, let  $\|\Mat M\|$ be the Frobenius norm of $\Mat M$ defined by $\|\Mat M\|^2=Tr(\tr{\Mat M}\Mat M)$, where $Tr$ is the trace operator. %For a vector $\Vect x$, $|\Vect x|$ is simply its euclidean norm.

%$\mathcal{B}(\Vect x,\rho):=\{\Vect y,\ |\Vect y-\Vect x|<\rho\}$. $\mathbbm{B}_i(\rho) = \mathcal{B}(i,\rho) \cap \ZZ[d]$.

The space $\RR[d]$ is endowed with the Borel $\sigma$-algebra $\mathcal{B}(\RR[d])$ and the Lebesgue measure $\lambda$. Let $\mSp$ be a measurable space, which aims at being the mark space, endowed with the $\sigma$-algebra $\mathcal M$ and the probability measure $\lambda^{\mE}$. The state space of the point processes will be $\sSp:=\RR[d]\times\mSp$ measured by $\mu:=\lambda\otimes\lambda^{\mE}$.
We shall denote for short $x^m=(x,m)$ an element of $\sSp$. 

%The space $\RR[d]$ is endowed with the Borel $\sigma$-algebra and the Lebesgue measure $\lambda$.
%  $\mSp$  $\lambda^{\mE}$. The state space of the point processes will be $\sSp:=\RR[d]\times\mSp$ measured by $\mu:=\lambda\otimes\lambda^{\mE}$.

%$x^m=(x,m)$ 

%The space of marked point configurations will be denoted by $\Omega=\Omega(\sSp)$. This is the set of locally finite subset $\varphi$ in $\sSp$ which means that for every bounded set $\Lambda\in \sB(\RR[d])$ the set $\varphi_\Lambda:=\varphi\cap(\Lambda\times \mSp)$ has finite cardinality $|\varphi_\Lambda|$. $\Omega$ is endowed with the $\sigma$-algebra $\mathcal F$ generated by the sets $\{\varphi\in\Omega,  |\varphi\cap(\Lambda\times A)|=n\}$ for all $n\in\NN$, for all  $A\in\mathcal M$ and for all $\Lambda\in\sB(\RR[d])$. We will use without ambiguity the notations, e.g. $\varphi\cup x^m:=\varphi\cup\{x^m\}$ and 
%for $x^m\in\varphi$, $\varphi\smallsetminus x^m:=\varphi\smallsetminus \{x^m\}$. 

A configuration is a subset $\varphi$ of $\sSp$ which is locally finite in that $\varphi_\Lambda:=\varphi \cap (\Lambda\times\mSp)$ has finite cardinality $N_\Lambda(\varphi):=|\varphi_\Lambda|$ for all $\Lambda \Subset \RR[d]$. The space $\Omega=\Omega(\sSp)$ of all configurations is equipped with the $\sigma$-algebra $\mathcal{F}$ that is generated by the counting variables $\varphi\to |\varphi_{\Lambda\times \mathcal A}|$ for any $\Lambda \Subset \RR[d]$ and any $\mathcal A\in \mathcal M$. Finally, let $T=\left( \tau_x \right)_{x\in \RR[d]}$ be the shift group, where $\tau_x:\Omega\to \Omega$ is the translation by the vector $-x \in \RR[d]$ (i.e. the application $\varphi\mapsto \cup_{(y,m)\in\varphi} \{(y-x,m)\}$). For the sake of simplicity, we set $\varphi\cup x^m:=\varphi\cup\{x^m\}$ and   $\varphi\setminus x^m:=\varphi\setminus \{x^m\}$. 

\subsection{Marked Gibbs point processes}

Our results will be expressed for general stationary Gibbs point processes. Since we are interested in asymptotic properties, we have to consider these point processes acting on the infinite volume $\RR[d]$. Let us briefly recall their definition.

A marked point process $\Phi$ is an $\Omega$-valued random variable, with probability distribution $P$ on $(\Omega,\mathcal F)$. The most prominent marked point process is the marked Poisson process $\pi^{z}$ with intensity measure $z\lambda\otimes\lambda^{\mE}$ on $\sSp$, with $z>0$ (see for example \cite{B-KKM78} for definition and properties). For $\Lambda\Subset \RR[d]$, let us denote  by $\pi^z_\Lambda$ the marginal probability measure in $\Lambda$ of the Poisson process with intensity $z$. Without loss of generality, the intensity  $z$ is fixed to~1 and we simply write $\pi$ and $\pi_\Lambda$ in place of $\pi^1$ and $\pi_\Lambda^1$.

Let $\ParV \in \RR[p]$ (for some $p\geq 1$). For any $\Lambda\Subset \RR[d]$, let us consider the parametric function $V_{\Lambda}(.;\ParV)$ from $\Omega$ into $\RR[]\cup\{+\infty\}$. From a physical point of view, $V_{\Lambda}(\varphi;\ParV)$ is the energy of $\varphi_{\Lambda}$ in $\Lambda$ given the outside configuration $\varphi_{\Lambda^c}$. In this paper, we focus on stationary point processes on $\RR[d]$, i.e. with stationary (i.e. $T$-invariant) probability measure. For any $\Lambda\Subset \RR[d]$, we therefore consider $V_{\Lambda}(.;\ParV)$ to be $T$-invariant, i.e. $V_{\Lambda}(\tau_x \varphi;\ParV)=V_{\Lambda}(\varphi;\ParV)$ for any $x\in\RR[d]$. 
Furthermore, ($V_{\Lambda}(.;\ParV))_{\Lambda\Subset\RR[d]}$ is a compatible family of energies, i.e.  for every $\Lambda\subset\Lambda'\Subset\RR[d]$, there exists a measurable function $\psi_{\Lambda,\Lambda'}$ from $\Omega$ into $\RR[]\cup\{+\infty\}$ such that 
\begin{equation}\label{compatibility}
 \forall\varphi\in\Omega\quad V_{\Lambda'}(\varphi;\ParV)=V_{\Lambda}(\varphi;\ParV)+\psi_{\Lambda,\Lambda'}(\varphi_{\Lambda^c};\ParV).
 \end{equation}
 
 A stationary marked Gibbs point process is characterized by a stationary marked Gibbs measure usually defined as follows (see \cite{B-Pre76}).
 \begin{definition}A probability measure $P_{\ParV}$ on $\Omega$ is a stationary marked Gibbs measure for the compatible family of $T$-invariant energies $(V_{\Lambda}(.;\ParV))_{\Lambda\Subset\RR[d]}$ if for every $\Lambda\Subset\RR[d]$, for $P_{\ParV}$-almost every outside configuration $\varphi_{\Lambda^c}$, the law of $P_{\ParV}$ given $\varphi_{\Lambda^c}$ admits the following conditional density with respect to $\pi_\Lambda$:
 $$f_{\Lambda}(\varphi_{\Lambda}|\varphi_{\Lambda^c};\ParV)=\frac{1}{Z_{\Lambda}(\varphi_{\Lambda^c};\ParV)}e^{-V_{\Lambda}(\varphi;\ParV)},$$
where $Z_{\Lambda}(\varphi_{\Lambda^c};\ParV)$ is a normalization called the partition function.
\end{definition}

The existence of a Gibbs measure on $\Omega$ which satisfies these conditional specifications is a difficult issue. We refer the interested reader to \cite{B-Rue69}, \cite{B-Pre76}, \cite{A-BerBilDro99}, \cite{A-Der05}, \cite{A-DerDroGeo09} for the technical and mathematical development of the existence problem.

In a first step, we assume that the family of energies is hereditary (the non-hereditary case will be considered in Section~\ref{non-hereditary}), which means that for any $\Lambda\Subset \RR[d]$, for any $\varphi\in\Omega$, and for all $x^m\in\Lambda\times\mSp$, 
\begin{equation}\label{heredite}V_{\Lambda}(\varphi;\ParV)=+\infty \Rightarrow V_{\Lambda}(\varphi \cup x^m;\ParV)=+\infty,\end{equation}
or equivalently, for all $x^m\in\varphi_{\Lambda}$,  $f_{\Lambda}(\varphi_{\Lambda}|\varphi_{\Lambda^c};\ParV)>0 \Rightarrow f_{\Lambda}(\varphi_{\Lambda}\setminus x^m|\varphi_{\Lambda^c};\ParV)>0$.

The minimal assumption of our paper is then:
\begin{list}{}{}
\item \textbf{[Mod]}: For any $\ParV \in\SpPar$, where $\SpPar$ is a compact subset of $\RR[p]$, there exists a stationary marked Gibbs measure $P_{\ParV}$ for the  compatible $T$-invariant hereditary family $(V_{\Lambda}(.;\ParV))_{\Lambda\Subset \RR[d]}$. Our data consist in the realization of a marked point process $\Phi$ with stationary marked Gibbs measure $P_{\ParVT}$, where $\ParVT \in \mathring{\SpPar}$ is the unknown parameter vector to estimate.
\end{list}
Let us note that \textbf{[Mod]} ensures the existence of at least one stationary Gibbs measure. When this Gibbs measure is not unique, we say that the phase transition occurs. In this situation the set of Gibbs measures is a  Choquet simplex  and any Gibbs measure is a mixture of extremal ergodic Gibbs measures. If the Gibbs measure is unique, it is necessary ergodic (see \cite{B-Geo88} for more details about these properties).

In the rest of this paper, the reader has mainly to keep in mind the concept of  local energy defined as the energy required to insert a point $x^m$ into the configuration $\varphi$ and expressed for any $\Lambda\ni x$ by 
 \begin{equation}\label{energielocale}
V^{}\left( x^{m}|\varphi; \ParV \right):=V_{\Lambda}(\varphi \cup x^m;\ParV)-V_{\Lambda}(\varphi;\ParV).
\end{equation}
From the compatibility of the family of energies, i.e. (\ref{compatibility}), this definition does not depend on $\Lambda$. 

%The main assumption involving the existence of Gibbs measure is the following.
%\begin{list}{}{}
%\item \textbf{[Mod]}: For any $\ParV \in \SpPar$,  the  compatible family of energies $(V_{\Lambda}(.;\ParV))_{\Lambda\in\sB(\RR[d])}$ is hereditary, invariant by translation, and such that an associated Gibbs measure $P_{\ParV}$ exists and is stationary. Our data consist in the realization of a point process with Gibbs measure $P_{\ParVT}$. The vector $\ParVT$ is thus the true parameter to be estimated, assumed to be in $\mathring{\SpPar}$.
%\end{list}

\section{The Takacs-Fiksel estimation procedure}\label{TFprocedure}

\subsection{Presentation}

The basic ingredient for the definition of the Takacs-Fiksel method is the so-called GNZ formula. This equation was proved by Georgii, Nguyen and Zessin in the seventies but other authors such as Papangelou and Takahashi also contributed to its establishment. See \cite{papangelou09} and \cite{A-Zessin09} for historical comments and \cite{A-Geo76} or \cite{A-NguZes79b} for a general presentation.
\begin{lemma}[Georgii-Nguyen-Zessin Formula]
Under \textbf{[Mod]}, for any measurable function $h(\cdot,\cdot;\ParV): \sSp\times \Omega\to \RR[]$  such that the following quantities are defined and finite, then
\begin{equation}\label{GNZnonstat}
\Esp\left( \ism[ \mathbb{R}^d \times \mSp]{h\left(x^m,\Phi;\ParV\right) e^{- \VIPar{x^m}{\Phi}{\ParVT}}} \right) = 
\Esp\left( \sum_{x^m \in \Phi } h\left(x^m,\Phi\setminus x^m;\ParV\right)  \right) ,
\end{equation}
where $\Esp$ denotes the expectation with respect to $P_{\ParVT}$. 
\end{lemma}
For stationary marked Gibbs point processes, (\ref{GNZnonstat}) reduces to
\begin{equation}\label{GNZ}
\Esp\left( h\left(0^M,\Phi;\ParV\right) e^{- \VIPar{0^M}{\Phi}{\ParVT}} \right) = 
\Esp\left( h\left(0^M,\Phi\setminus 0^M;\ParV\right)  \right),
\end{equation}
where $M$ denotes a random variable with probability distribution $\mm$. \\

A second tool used throughout the paper is the ergodic Theorem established in \cite{A-NguZes79}. Let us give a simpler form, here, which is sufficient in this paper.

\begin{lemma}[Ergodic result]\label{ergodiclemma}
Under \textbf{[Mod]}, we assume that $P_{\ParVT}$ is ergodic. Then for any family of measurable functions $F_\Lambda$, indexed by the bounded sets $\Lambda$, from $\Omega$ to $\RR[]$ which are additive (i.e. $F_{\Lambda\cup \Lambda'}=F_\Lambda+F_{\Lambda'}-F_{\Lambda\cap \Lambda'}$), shift invariant (i.e. $F_\Lambda(\varphi)=F_{\tau(\Lambda)}(\tau(\varphi))$ for any translation $\tau$) and integrable (i.e. $\Esp(|F_{[0,1]^d}|)<+\infty$), we have that for $P_{\ParVT}$-almost every $\varphi$
$$ \lim_{n\to+\infty} |\Lambda_n|^{-1}F_{\Lambda_n}(\varphi)=\Esp(F_{[0,1]^d}),$$
where $\Lambda_n=[-n,n]^d$ (other regular domains $(\Lambda_n)_{n\geq 1}$ converging towards $\RR[d]$ could be also considered).
\end{lemma}

Let $h(\cdot,\cdot;\ParV): \sSp\times \Omega\to \RR[]$ and let us define for any $\varphi \in \Omega$, $\ParV\in \SpPar$ and $\Lambda\Subset\RR[d]$
\begin{equation} \label{eq-defC}
C_{\Lambda}(\varphi;h,\ParV) := \ism[\Lambda \times \mSp]{ h(x^m,\varphi;\ParV) e^{-V(x^m|\varphi;\ParV)} } -\sum_{x^m\in \varphi_\Lambda} h(x^m,\varphi\setminus x^m;\ParV).
\end{equation}

Assume that we observe the realization of a marked point process $\Phi$ satisfying \textbf{[Mod]} in a domain $\Lambda_n$. For appropriate choices of the functional $h$ and a sequence of domain $\Lambda_n$, then it is possible to apply the ergodic result in Lemma \ref{ergodiclemma} to prove that the first and second terms of $|\Lambda_n|^{-1} C_{\Lambda_n}(\Phi;h,\ParV)$ respectively converge $P_{\ParVT}$-almost surely to the left and right terms of~(\ref{GNZ}).

The latter observation is the basic argument to define the Takacs-Fiksel method. Let us give $K$ functions $h_k(\cdot,\cdot;\ParV): \sSp\times \Omega\to \RR[]$
(for $k=1,\ldots,K$), then the Takacs-Fiksel estimator is simply defined by

\begin{equation}\label{eq-defTakacs}
\widehat{\ParV}_n(\varphi) := \widehat{\ParV}_n^{TF}(\varphi)= \mathop{\arg\min}_{\ParV\in \SpPar} \sum_{k=1}^K C_{\Lambda_n}(\varphi;h_k,\ParV)^2,
\end{equation}
where $\mathop{\arg\min}_{\ParV\in \SpPar} F(\ParV)$ means the parameter $\ParV$ which minimizes the function $F$. Under identification assumption (\ref{identifiabilite}) given later, this minimum is obtained for a unique $\ParV$ provided that $n$ is large enough.

\subsection{Some examples}\label{sec-examples}

In this section, some examples of models and test functions $h$, involved in \eqref{eq-defC}, are provided. The choices made in previous studies are presented in \ref{sec-classical}. The two examples presented in \ref{sec-quick} and \ref{sec-strauss} show the relevance of the Takacs-Fiksel procedure to provide quick estimates. The last example, in \ref{sec-quermass}, is concerned with the possible identification of a special marked point process, the quermass model, and shows that appropriate choices of test functions can solve identification problems when points are not observed.

There is no asymptotic consideration in this section. Therefore, the different estimates are defined over a window, say $\Lambda$, and are denoted by $\widehat{\theta}$.

\subsubsection{Classical examples}\label{sec-classical}

Let us first quote the particular case when the Takacs-Fiksel estimator reduces to the maximum pseudo-likelihood estimator introduced by \cite{A-JenMol91} for spatial point processes. The MPLE is obtained by maximizing the log-pseudo-likelihood contrast function, given by
\begin{equation} \label{eq-defLPL}
LPL_{\Lambda}(\varphi;\ParV)  = -\ism[\Lambda \times \mSp]{e^{-V(x^m|\varphi;\ParV)} } \; - \;  \sum_{x^m\in \varphi_{\Lambda}} V(x^m|\varphi\setminus x^m;\ParV).
\end{equation}
Therefore, with the choice $h_k(x^m,\varphi;\ParV) = \frac{\partial}{\partial\theta_k} V(x^m|\varphi;\ParV)$, $k=1,\dots,p$, the estimator $\widehat{\ParV}$, defined by \eqref{eq-defTakacs}, solves the system  $C_{\Lambda}(\varphi;h_k,\ParV)=0$, $k=1,\dots,p$, which means that $\widehat{\ParV}$ is the root of the gradient vector of $LPL_{\Lambda}$, i.e. $\widehat{\ParV}$ is the MPLE.

The first empirical study of the Takacs-Fiksel estimator can be found in \cite{A-Diggle94}, where this estimate is compared to other estimators for some unmarked pairwise interaction point processes with two parameters. In this context different test functions $h_{r_1},\cdots,h_{r_K}$ were used, where for  $r>0$ 
 $$h_r(x,\varphi;\ParV)=\left|\varphi_{\mathcal B(x,r)}\right|=\sum_{y\in\varphi} \mathbf{1}_{[0, r]}(\|y-x\|).$$
The integral term involved in \eqref{eq-defC} is approximated by discretization and the induced estimation \eqref{eq-defTakacs} is then assessed. Note that with this choice of test functions, the sum term in \eqref{eq-defC}, when normalized by $|\Lambda|^{-1}$,  is an estimation of $\rho^2\mathcal K(r)$, where $\mathcal K(\cdot)$ is the reduced second order function and $\rho$ denotes the intensity of the stationary point process $\Phi$, i.e. for all $B\in\sB(\RR[d])$, $\Esp(\Phi(B))=\rho|B|$.
 
The latter choice requires the computation of the integral in \eqref{eq-defC} for all $\ParV$. A more convenient choice could be the one first proposed by Fiksel:
  \begin{equation}\label{fiksel-choice}h_r(x,\varphi;\ParV)=\left|\varphi_{\mathcal B(x,r)}\right| e^{ \VIPar{x}{\varphi}{\ParV}}=e^{ \VIPar{x}{\varphi}{\ParV}}\sum_{y\in\varphi} \mathbf{1}_{[0, r]}(\|y-x\|).\end{equation}
 In the stationary case, this leads to the following approximation thanks to the ergodic theorem
 \begin{equation}\label{fiksel-integral}\frac{1}{|\Lambda|}\int_{\Lambda}{ h_r(x,\varphi;\ParV) e^{-V(x|\varphi;\ParV)} dx}\approx \Esp\left| \Phi_{\mathcal B(0,r)}\right| = \rho \pi r^2.
\end{equation}
The integral term in  \eqref{eq-defC} is thus easily approximated by $|\varphi_{\Lambda}|\pi r^2$ for all $\ParV$, while the sum can be explicitly computed. These historical choices of test functions have a natural extension in the marked case.

 \subsubsection{Some choices leading to quick estimations}\label{sec-quick}
 
The main advantage of the Takacs-Fiksel procedure is to provide quick consistent estimators, that might supply initial values for a more evolved procedure.  A simple way to achieve this goal is to generalize \eqref{fiksel-choice} and consider test functions of the form 
$$h(x^m,\varphi;\ParV)=\tilde h(x^m,\varphi)e^{ \VIPar{x^m}{\varphi}{\ParV}},$$
where $\tilde h(x^m,\varphi)$ does not depend on $\ParV$. So, the integral term in  \eqref{eq-defC} has to be computed only once and not for all $\ParV$, while the sum term in  \eqref{eq-defC} does not require any approximation.  Hence, the optimisation problem \eqref{eq-defTakacs} may be resolved very quickly.

In some particular examples, explicit formulas may even be obtained for the integral term, as in \eqref{fiksel-integral}. In the same spirit, an explicit estimator for the Strauss interaction is provided below.

  \subsubsection{An example of explicit estimator for the Strauss process}\label{sec-strauss}
 
The (non-marked) Strauss process with range of interaction $R>0$ is given  for any $\Lambda\Subset\RR[d]$ by
\begin{equation} \label{strauss}  
V_{\Lambda}\left( \varphi ; \ParV \right) = \theta_1 |\varphi_{\Lambda}| + \theta_2\sum_{\begin{subarray}{l} x,y \in \varphi\\ \{x,y\}\cap\Lambda\not=\emptyset \end{subarray}} \mathbf{1}_{[0,R]} \left( \|x-y\|\right),
\end{equation}
   where $\theta_1\in\RR[]$ and $\theta_2>0$ are the two parameters of the model.
Alternatively,
$$ \VIPar{x}{\varphi}{\ParV}= \theta_1 + \theta_2\sum_{y \in \varphi} \mathbf{1}_{[0, R]}(\|y-x\|).$$
Let us consider the following family of test functions, for $k\in \NN\setminus \{0\}$, 
\begin{equation} \label{def-hkStrauss}
h_{k}(x,\varphi;\ParV)=\begin{cases} e^{(k-1)\theta_2} &\textrm{if } \left|\varphi_{\mathcal B(x,R)}\right|=k-1,\\ 0 & \textrm{otherwise.}\end{cases}
\end{equation}
This choice  gives in (\ref{eq-defC})
$$ C_{\Lambda}(\varphi;h_k,\ParV) = e^{-\theta_1}  V_{k-1,\Lambda}(\varphi)-e^{(k-1)\theta_2}  N_{k-1,\Lambda}(\varphi),$$
where $N_{k,\Lambda}(\varphi)$ denotes the number of points $x\in\varphi_{\Lambda}$ such that $|\mathcal B(x,R)\cap(\varphi\setminus x)|=k$  and $V_{k,\Lambda}(\varphi)$ denotes the volume of the set $\{y\in\Lambda, |\mathcal B(y,R)\cap\varphi|=k\}$.

Several explicit estimators may be obtained following (\ref{eq-defTakacs}) from (at least) two test functions as above. Let us quote the simplest one, corresponding to the choice $h_1$ and $h_2$  in (\ref{eq-defTakacs}). This leads to the contrast function $C_{\Lambda}(\varphi;h_1,\ParV)^2+C_{\Lambda}(\varphi;h_2,\ParV)^2$ which  vanishes at the unique point $(\widehat{\ParV}_{1}(\varphi),\widehat{\ParV}_{2}(\varphi))$ with 
\begin{equation}\label{TFEexplicit}
\widehat{\ParV}_{1}(\varphi)=\ln\left(\frac{V_{0,\Lambda}(\varphi)}{N_{0,\Lambda}(\varphi)}\right),\quad \widehat{\ParV}_{2}(\varphi)=\ln\left(\frac{V_{1,\Lambda}(\varphi)}{N_{1,\Lambda}(\varphi)}\right)-\ln\left(\frac{V_{0,\Lambda}(\varphi)}{N_{0,\Lambda}(\varphi)}\right).
\end{equation}
This estimator of $(\theta_1,\theta_2)$ is completely explicit, provided the quantities $N_{k,\Lambda}(\varphi)$ and $V_{k,\Lambda}(\varphi)$ are available. They can be easily approximated by computational geometry tools.

 \subsubsection{A  solution for unobservability issues}\label{sec-quermass}
 
 The quermass model introduced in \cite{A-Kendall99} is a marked point process which aims at modelling random sets in $\RR[2]$. This is a generalization of the well-known Boolean model  to interacting random balls. 
Let us denote by $x^R$ a marked point where $x$ and $R>0$ (i.e. the mark) respectively represent the center and the radius of the associated ball $\mathcal B(x,R)$. For a finite configuration $\varphi$, i.e. with a finite support instead of $\RR[2]$, the quermass energy is defined for $(\theta_1,\theta_2,\theta_3,\theta_4)\in\RR[4]$ by:
 $$V\left( \varphi ; \ParV \right) = \theta_1 |\varphi| + \theta_2\ \mathcal P(\Gamma)+\theta_3\ \mathcal A(\Gamma)+\theta_4\ \mathcal E(\Gamma)\quad\textrm{where } \Gamma=\bigcup_{(x,R)\in\varphi}\mathcal B(x,R)$$
and $\mathcal P(\Gamma)$, $\mathcal A(\Gamma)$ and $\mathcal E(\Gamma)$ denote respectively the perimeter, the area and the Euler-Poincar\'e characteristic (i.e. number of components minus number of holes) of the set $\Gamma$. To extend this definition to the infinite support $\RR[2]$, it is convenient to suppose that the radii of the balls are almost surely uniformly bounded (i.e. $\lambda^{\mE}([0,R_0])=1$ for some $R_0>0$). In this case the family of energies $(V_\Lambda)$ is defined by 

$$ V_\Lambda\left( \varphi ; \ParV \right) = V\left(\varphi_{\Lambda\oplus B(0,2R_0)}; \ParV \right)-V\left(\varphi_{\Lambda\oplus B(0,2R_0)\backslash \Lambda}; \ParV \right).$$

This definition may be extended to unbounded radius, though a restriction to the so-called tempered configurations is needed to ensure the existence of the associated Gibbs measure.  We refer to \cite{A-Der09} for more details.

 When $\theta_2=\theta_3=\theta_4=0$, this model reduces to the Boolean model (see \cite{B-StoKenMecRus87} for a survey). The area process (see \cite{A-Baddeley95}) is also a particular case, taking $\theta_2=\theta_4=0$.

In practice, one only observes the random set $\Gamma$, so the marked points $x^R$ in $\varphi$ are unknown. A challenging task is then to estimate the parameters $(\theta_1,\theta_2,\theta_3,\theta_4)$ in the presence of this unobservability issue. In particular, a direct application of the maximum likelihood or  pseudo-likelihood method is impossible to estimate all the parameters, and especially $\theta_1$ which requires the observation of the number of points in $\varphi$. For the other parameters, which are related to the observable functionals $\mathcal P$, $\mathcal A$ and $\mathcal E$, the MLE has been investigated in \cite{A-Moller08}. 

Let us show that the Takacs-Fiksel procedure may be used to estimate $\theta_1$ in spite of this unobservability issue. Indeed,  it is possible to choose some test function $h$ such that both the integral and the sum in \eqref{eq-defC} are computable. The unobservability issue occurs mainly for the sum term. Let us consider the following example of test function:
\begin{equation}\label{def-hper}
h_{per}(x^R, \varphi;\ParV)=\mathcal P\left(\mathcal C(x,R)\cap\Gamma^c\right),
\end{equation}
where $\mathcal C(x,R)$ is the sphere $\{y, ||y-x||=R\}$. For any finite configuration $\varphi$, we then  have $$\sum_{x^R\in \varphi} h_{per}(x^R,\varphi\setminus x^R;\ParV) = \mathcal P(\Gamma),$$
so that this sum is computable even if each term $ h_{per}(x^R,\varphi\setminus x^R;\ParV)$ is not. If the configuration $\varphi$ is infinite then for any bounded set $\Lambda$, $\sum_{x^R\in \varphi_\Lambda} h_{per}(x^R,\varphi\setminus x^R;\ParV) $ is equal to the perimeter of $\Gamma$ restricted to $\Lambda$ plus a boundary term which is asymptotically negligible with respect to the volume of $\Lambda$.

 Consequently, assuming $(\theta_2, \theta_3, \theta_4)$ known, $\theta_1$ may be estimated thanks to (\ref{eq-defTakacs}) with the above choice of test function. 

The joint estimation of all four parameters might be achieved thanks to additional test functions sharing the same 
property as  above, i.e. such that the sum in  \eqref{eq-defC}  is observable.

In the particular case of the area process, $\theta_2=\theta_4=0$ and $R$ is constant (i.e. $\lambda^{\mE}=\delta_R$), it suffices to find one more test function to ensure an identifiable estimation (see Example~\ref{exArea} in Section \ref{SectionIdent} for more details). A possible additional test function is 
\begin{equation}\label{def-hiso}
h_{iso}(x^R, \varphi;\ParV)=\begin{cases} 1\quad\text{if}\quad \mathcal P\left(\mathcal C(x,R)\right)=2\pi R \\ 0 \quad\text{otherwise.}\end{cases}
\end{equation}
In this case $\sum_{x^R\in \varphi} h_{iso}(x^R,\varphi\setminus x^R;\ParV)$ corresponds to the number of isolated balls in $\Gamma$.

%%%%%%%%%%%%%%%%%%%%%%%%%%%%%%%%%%%%%%%%%%%%%%%%%%%%%%%%%%%%%%%%%%%%%%%%%%%%%%%%%%%%%%%%%%%
%%%%%%%%%%%%%%%%%%%%%%%%%%%%%% Main Results %%%%%%%%%%%%%%%%%%%%%%%%%%%%%%%%%%%%%%%%%%%%%%%
%%%%%%%%%%%%%%%%%%%%%%%%%%%%%%%%%%%%%%%%%%%%%%%%%%%%%%%%%%%%%%%%%%%%%%%%%%%%%%%%%%%%%%%%%%%

\section{Asymptotic results for the Takacs-Fiksel estimator} \label{sec-results}

We present in this section asymptotic results for the Takacs-Fiksel estimator for a point process  satisfying \textbf{[Mod]}  and assumed to be observed in a domain $\Lambda_n$, where $(\Lambda_n)_{n\geq 1}$ is a sequence of increasing cubes whose size goes to $+\infty$ as $n$ goes to $+\infty$.

First, for a function $g$ depending on $\ParV$, we denote by $\Vect{g}^{(1)}(\ParV)$ (resp. $\Mat{g}^{(2)}(\ParV)$) the gradient vector of length $p$ (resp. the Hessian matrix of size $(p,p)$) evaluated at $\ParV$. Let us rewrite the Takacs-Fiksel estimator as 
$$
\widehat{\ParV}_n(\varphi) = \mathop{\arg\min}_{\ParV\in \SpPar} U_{\Lambda_n}(\varphi;\Vect{h},\ParV),
$$
with $U_{\Lambda_n}(\varphi;\Vect{h},\ParV) = |\Lambda_n|^{-2} \; \sum_{k=1}^K C_{\Lambda_n}(\varphi;h_k,\ParV)^2$, where $\Vect{h}=(h_1,\dots,h_K)$ and $C_{\Lambda_n}$ is given by \eqref{eq-defC}.

\subsection{Consistency}

The consistency is obtained under the following assumptions, denoted by \textbf{[C]}: for any Gibbs measure $P_{\ParVT}$, for all $\ParV \in \SpPar$, $k=1,\ldots,K$ and $\varphi\in \Omega$
\begin{itemize}
\item[\textbf{[C1]}] For all $x\in\RR[d]$, $h_k(x^m,\varphi;\ParV)=h_k(0^m,\tau_x\varphi;\ParV)$ and
$${\displaystyle \Esp \left( \left| h_k \left( 0^M , \Phi ; \ParV \right) \right| e^{- V^{}\left( 0^M|\Phi; \ParV^\prime \right)}\right) <+\infty, \mbox{ for }  \ParV^\prime=\ParV, \ParVT.}$$
\item[\textbf{[C2]}] $U_{\Lambda_n}(\varphi;\Vect{h},\cdot)$ is a continuous function for $P_{\ParVT}-$a.e. $\varphi$.
\item[\textbf{[C3]}] 
\begin{equation}\label{identifiabilite}
\sum_{k=1}^K \Esp \left( h_k(0^M,\Phi;\ParV) \left(
e^{-V(0^M|\Phi;\ParV)} -e^{-V(0^M|\Phi;\ParVT)}
\right) \right)^2 = 0 \quad \Longrightarrow \ParV= \ParVT.
\end{equation}
\item[\textbf{[C4]}] $h_k$ and $f_k$, defined by $f_k(x^m,\varphi;\ParV):= h_k(x^m,\varphi;\ParV)e^{-V(x^m|\varphi;\ParV)}$, are continuously differentiable and
$$\Esp\left(
\max_{\ParV\in \SpPar} |f_k(0^M,\Phi;\ParV)|
 \right)<+\infty 
\quad \mbox{ and } \quad
\Esp\left(
\max_{\ParV\in \SpPar} |h_k(0^M,\Phi;\ParV)| e^{-V(0^M|\Phi;\ParVT)}
 \right)<+\infty,$$
$$\Esp\left( 
\max_{\ParV\in \SpPar} \| \Vect{f_k}^{(1)}(0^M,\Phi;\ParV)\|
 \right)<+\infty 
\quad \mbox{ and } \quad
\Esp\left(
\max_{\ParV\in \SpPar} \| \Vect{h_k}^{(1)}(0^M,\Phi;\ParV)\|
e^{-V(0^M|\Phi;\ParVT)}
\right)<+\infty.$$
\end{itemize}

\begin{theorem} \label{prop-cons}
Assuming~\textbf{[Mod]} and \textbf{[C]} then, as $n \to +\infty$, the Takacs-Fiksel estimator $\widehat{\ParV}_n(\varphi)$ converges towards $\ParVT$ for $P_{\ParVT}-$a.e. $\varphi$.
\end{theorem}

%\begin{remark}
Assumptions \textbf{[C1]}, \textbf{[C2]} and \textbf{[C4]} are related to the regularity and the integrability of the different test functions and the local energy function. Some general criteria may be proposed to verify these assumptions, see Section~\ref{sec-ass} for a discussion. Assumption \textbf{[C3]} corresponds to an identifiability condition and requires much more attention. It is well-known that such an assumption is fulfilled when $\Vect{h}=\Vect{V}^{(1)}$ (leading to the MPLE) under mild assumptions (see Assumption \textbf{[Ident]} proposed by \cite{A-BilCoeDro08}). The question to know if this remains true for more general test functions is difficult (actually it is untrue in several cases). This will be discussed specifically in Section~\ref{SectionIdent}.
%\end{remark}

\subsection{Asymptotic normality}

We need the following assumptions denoted by \textbf{[N]}: For any Gibbs measure $P_{\ParVT}$, $k=1,\ldots,K$, $\Lambda\Subset \mathbb{R}^d$, $\varphi\in \Omega$  and $\ParV$ in a neighborhood $\mathcal{V}(\ParVT)$ of $\ParVT$: 
\begin{itemize}
\item[\textbf{[N1]}] $\Esp \left( \left|  C_\Lambda(\Phi;h_k,\ParVT)  \right|^3 \right)<+\infty$.
\item[\textbf{[N2]}] For any sequence of bounded domains $\Gamma_n$ such that $\Gamma_n \to 0$ as $n\to +\infty$, $\Esp\left( C_{\Gamma_n}(\Phi;h_k,\ParVT)^2\right)\to 0$.
\item[\textbf{[N3]}] $C_\Lambda(\varphi;h_k,\ParVT)$ depends only on $\varphi_{\Lambda\oplus B(0,D)}$ for some $D\geq0$ (which is uniform in $\Lambda$, $\varphi$, $\ParVT$).
\item[\textbf{[N4]}] $h_k$ and $f_k$ (defined in \textbf{[C4]}) are twice continuously differentiable in $\ParV$ and 
$$
\Esp \left(\| \Mat{h}^{(2)}(0^M,\Phi;\ParV)\|  e^{-\VIPar{0^M}{\Phi}{\ParVT}} \right) <+\infty 
\quad \mbox{ and } \quad
\Esp \left(\| \Mat{f}^{(2)}(0^M,\Phi;\ParV)\|  \right) <+\infty .
$$
\end{itemize}
Let us remark that Assumption \textbf{[N3]} leads us to consider in general that $V$ has a finite range, which means that there exists $D\geq 0$ such that for all $(m,\varphi) \in \mSp\times \Omega$ and all $\ParV \in \SpPar$
$$\VIPar{0^{m}}{\varphi}{\ParV} = \VIPar{0^{m}}{\varphi_{\mathcal B(0,D)}}{\ParV}.$$
The same kind of finite range property is also expected for $(h_k)$.

\begin{theorem} \label{prop-norm}
Under Assumptions \textbf{[Mod]}, \textbf{[C]} and \textbf{[N]}, for any ergodic Gibbs measure $P_{\ParVT}$ the following convergence in distribution holds as $n\to +\infty$
\begin{equation} \label{conv-norm}
|\Lambda_n|^{1/2} \Mat{\mathcal{E}}(\Vect{h},\ParVT) \tr{\Mat{\mathcal{E}}(\Vect{h},\ParVT)} \left(\widehat{\ParV}_n(\Phi)-\ParVT \right) \stackrel{d}{\rightarrow} \mathcal{N}\left(0, {\Mat{\mathcal{E}}(\Vect{h},\ParVT)} \Mat{\Sigma}(\Vect{h},\ParVT) \tr{\Mat{\mathcal{E}}(\Vect{h},\ParVT)} \right),
\end{equation}
where $\Mat{\mathcal{E}}(\Vect{h},\ParVT)$ is the $(p,K)$ matrix defined for $i=1,\ldots,p$ and $k=1,\ldots,K$ by
$$
\left( \Mat{\mathcal{E}}(\Vect{h},\ParVT) \right)_{ik}= \Esp \left( 
h_k(0^M,\Phi;\ParVT) \left(\Vect{V}^{(1)}(0^M|\Phi;\ParVT)\right)_i e^{-V(0^M|\Phi;\ParVT)} 
\right)
$$
and where $\Mat{\Sigma}(\Vect{h},\ParVT)$ is the $(K,K)$ matrix defined by 
\begin{equation}\label{matC}
\Mat{\Sigma}(\Vect{h},\ParVT)= D^{-d} \; \sum_{|\ell|\leq 1} \Esp \left( 
\widetilde{\Vect{C}}_{\Delta_0(D)}(\Phi;\Vect{h},\ParVT) \tr{\widetilde{\Vect{C}}_{\Delta_\ell(D)}(\Phi;\Vect{h},\ParVT)}
\right),
\end{equation}
where, for all $k\in \ZZ[d]$, $\Delta_k(D)$ is the cube centered at $kD$ with side-length $D$ and where, for any bounded domain $\Lambda$, $\widetilde{\Vect{ C}}_{\Lambda}(\Phi;\Vect{h},\ParVT):=\left( {C}_{\Lambda}(\Phi;h_k,\ParVT)\right)_{k=1,\ldots,K}$.
\end{theorem}

\begin{remark}
While Assumptions \textbf{[N1-3]} will ensure that a central limit theorem holds for $\Vect{U}_{\Lambda_n}^{(1)}(\Phi;\Vect{h},\ParVT)$, Assumption \textbf{[N4]} is required to prove that $\Mat{U}_{\Lambda_n}^{(2)}(\Phi;\Vect{h},\ParV)$ is uniformly bounded in a neighborhood of $\ParVT$. These two statements allow us to apply a general central limit theorem for minimum contrast estimators ({\it e.g.} Theorem 3.4.5 of \cite{B-Guy95}).
\end{remark}

\begin{remark} \label{est-cov}
In Theorem \ref{prop-norm}, if the Gibbs measure $P_{\ParVT}$ is  stationary (and then not necessarily ergodic), it is a mixture of ergodic measures and the left hand term in (\ref{conv-norm}) converges in distribution to a mixture of normal distributions. 
We may also propose an asymptotic result valid in the presence of a phase transition. Indeed, if the matrix  ${\Mat{\mathcal{E}}(\Vect{h},\ParVT)} \Mat{\Sigma}(\Vect{h},\ParVT) \tr{\Mat{\mathcal{E}}(\Vect{h},\ParVT)}$ is positive definite (for any extremal ergodic measure of the Choquet simplex of stationary Gibbs measures), then we may define a consistent empirical version $S(\Phi)$ of $ \left[{\Mat{\mathcal{E}}(\Vect{h},\ParVT)} \Mat{\Sigma}(\Vect{h},\ParVT) \tr{\Mat{\mathcal{E}}(\Vect{h},\ParVT)}\right]^{-1/2}\Mat{\mathcal{E}}(\Vect{h},\ParVT) \tr{\Mat{\mathcal{E}}(\Vect{h},\ParVT)}$ (see Section~4 of \cite{A-CoeLav10} for more details) to obtain:
\begin{equation*} 
|\Lambda_n|^{1/2} S(\Phi) \left(\widehat{\ParV}_n(\Phi)-\ParVT \right) \stackrel{d}{\rightarrow} \mathcal{N}\left(0, \Mat I \right).
\end{equation*}

\end{remark}

\begin{remark}
Following Section~\ref{sec-classical}, let us underline that~(\ref{conv-norm}) is coherent with the asymptotic normality of the MPLE established in \cite{A-CoeDro09}, i.e. with the case $K=p$ and $\Vect{h}=\Vect{V}^{(1)}$. Indeed, with similar assumptions to the ones presented in the present paper, Equation (4.4) in Theorem~2 \cite{A-CoeDro09} states that 
\begin{equation}\label{norm-MPLE}
|\Lambda_n|^{1/2} \Mat{A}(\ParVT) \left(\widehat{\ParV}_n(\Phi)-\ParVT \right) \stackrel{d}{\rightarrow} \mathcal{N}\left(0, \Mat{\Sigma}(\Vect{V}^{(1)},\ParVT)  \right),
\end{equation}
where $\Mat{A}(\ParVT)$ is the symmetric $(p,p)$ matrix given for $i,k=1,\ldots,p$ by
$$
\left( \Mat{A}(\ParVT) \right)_{ik} = \Esp \left( \left(\Vect{V}^{(1)}(0^M|\Phi;\ParVT)\right)_k \left(\Vect{V}^{(1)}(0^M|\Phi;\ParVT)\right)_i   
e^{-V(0^M|\Phi;\ParVT)}
\right).
$$
Since $\Vect{h}=\Vect{V}^{(1)}$, then $\Mat{\mathcal{E}}(\Vect{V}^{(1)},\ParVT) := \Mat{A}(\ParVT)$ and therefore (\ref{conv-norm}) reduces to
$$
|\Lambda_n|^{1/2} \Mat{A}(\ParVT)^2 \left(\widehat{\ParV}_n(\Phi)-\ParVT \right) \stackrel{d}{\rightarrow} \mathcal{N}\left(0, \Mat{A}(\ParVT) \Mat{\Sigma}(\Vect{V}^{(1)},\ParVT) \Mat{A}(\ParVT) \right)
$$
which is exactly~(\ref{norm-MPLE}) by assuming that $\Mat{A}(\ParVT)$ is invertible.
\end{remark}

\begin{remark}
%Proposition~\ref{prop-norm} suggests the two following questions which are clearly very interesting from a practical point of view~:
%\begin{itemize}
%\item Can one derive data-driven estimates of  $\Mat{\mathcal{E}}(\Vect{h},\ParVT) \tr{\Mat{\mathcal{E}}(\Vect{h},\ParVT)}$ and  ${\Mat{\mathcal{E}}(\Vect{h},\ParVT)} \Mat{\Sigma}(\Vect{h},\ParVT) \tr{\Mat{\mathcal{E}}(\Vect{h},\ParVT)}$ in order to propose asymptotic confidence bands for example~?
%\item 
The question how to choose the test functions in order to minimize  the norm of the asymptotic covariance matrix is difficult to answer, still open and is a perspective for future work.
%\end{itemize}
%The first question may be resolved in two different ways. The first one consists in approximating these matrices by Monte-Carlo methods (by generating point processes with parameter vector $\widehat{\ParV}_n$) or by using empirical estimates. Empirical estimates have been proposed for the MPLE by \cite{A-BilCoeDro08} (Theorem~2 Equation~(4.6)). In \cite{A-CoeLav10} (in particular Section~4)  a more general result is provided for a similar problem and the same procedure may be applied in the present context. \\
%The second question is much more difficult, still open and is a perspective for future work.
\end{remark}

\subsection{Discussion} \label{sec-ass}

The present paragraph is devoted to the discussion of Assumptions \textbf{[Mod]},  \textbf{[C]} (except \textbf{[C3]}) and \textbf{[N]}. In the previous sections, we have expressed the different assumptions in a very general way. Our aim, here, is to make these assumptions concrete for a wide range of models and a wide range of test functions, in order to illustrate that our setting is not restrictive. In particular, we will focus on exponential family models having a local energy of the form:  
\begin{equation}\label{formeexp}
V(x^m|\varphi;\ParV):= \PrSc{\ParV}{\Vect{V}(x^m|\varphi)}=\ParV_1{V}_1(x^m|\varphi)+\ldots+\ParV_p{V}_p(x^m|\varphi),
\end{equation} 
with $
\Vect{V}=(V_1,\ldots, V_p)$ a vector function  from $\sSp\times \Omega(\sSp)$ to $\RR[p]$.

Let us consider the following assumptions: for all $(m,\varphi) \in \mSp\times \Omega$.

\begin{itemize}
\item[\textbf{[Exp]}] For $i=1,\cdots,p$, for all $x\in\RR[d]$, $V_i(x^m|\varphi)=V_i(0^m|\tau_x\varphi)$ and  there exist $\cSEx{\inf}{i},\cSEx{\sup}{i}\geq 0$, $k_i\in \NN, D>0$  such that one of the two following assumptions is satisfied :
$$
\Par_i \geq 0 \mbox{ and } 
-\cSEx{\inf}{i}\leq \SExI{i}{0^{m}}{\varphi}=\SExI{i}{0^{m}} {\varphi_{\mathcal{B}(0,D)}} \leq \cSEx{\sup}{i} |\varphi_{\mathcal{B}(0,D)} |^{k_i}.
$$
or
$$
-\cSEx{\inf}{i}\leq \SExI{i}{0^{m}}{\varphi}=\SExI{i}{0^{m}} {\varphi_{\mathcal{B}(0,D)}} \leq \cSEx{\sup}{i}.
$$
\item[$\widetilde{\mbox{\textbf{[Exp]}}}$] Assumption \textbf{[Exp]} with $k_i=0$ or 1 for all $i$ (when $\theta_i\geq 0$).
\item[\textbf{[H]}] For all $x\in\RR[d]$, $h(x^m,\varphi;\ParV)=h(0^m,\tau_x\varphi;\ParV)$ and there exist ${\kappa}>0$, ${k} \in \NN, D>0$ such that $h(0^m,\varphi;\ParV)=h(0^m,\varphi_{\mathcal{B}(0,D)};\ParV)$, such that $h(0^m,\varphi;\cdot)$ is twice continuously differentiable in $\ParV$ and such that $|Y(\varphi,m)| \leq {\kappa} |\varphi_{\mathcal{B}(0,D)}|^{{k}}$, where
\begin{equation}\label{eq-Y}
Y(\varphi,m):= \max\left( 
|h(0^m,\varphi;\ParV)|, \| \Vect{h}^{(1)}(0^m,\varphi;\ParV) \|, \| \Mat{h}^{(2)}(0^m,\varphi;\ParV) \|  \right).
\end{equation}
\item[$\widetilde{\mbox{\textbf{[H]}}}$] $h(0^m,\varphi;\ParV) = \widetilde{h}(0^m,\varphi;\ParV) e^{\PrSc{\ParV}{\Vect{V}(0^m|\varphi)}}$ with $\widetilde{h}$ satisfying \textbf{[H]}.
\end{itemize}
Let us underline that the different constants involved in these assumptions are assumed to be independent of $m,\varphi,\ParV$. Note also that if the test function $h$ is independent of $\ParV$, $Y(\varphi,m)$ obviously reduces to $|h(0^m,\varphi)|$. \\

These assumptions are  common and very simple to check. Assumption \textbf{[Exp]} has already been investigated in \cite{A-BilCoeDro08}. It includes a wide variety of models such as the overlap area point process, the multi-Strauss marked point process, the $k-$nearest-neighbor multi-Strauss marked point process, the Strauss type disc process, the Geyer's triplet point process, the area process, some special cases of quermass process (for instance when $\lambda^{\mE}$ has a compact support not containing $0$ and $\ParV_4=0$), etc. Among these models, the only one that does not satisfy $\widetilde{\mbox{\textbf{[Exp]}}}$ is the Geyer's triplet point process (see \cite{A-BilCoeDro08}, p.242).

On the other hand, the test functions $h(x^m,\varphi;\ParV)=1$,  $h(x^m,\varphi;\ParV)=\Vect{V}^{(1)}_k(x^m|\varphi;\ParV)=V_k(x^m|\varphi)$, $h(x^m,\varphi;\ParV)=|\varphi_{\mathcal{B}(x,r)}|$ satisfy \textbf{[H]} (for the second one, it is implied by \textbf{[Exp]}).  Note that the functional described by (\ref{def-hkStrauss}) for the Strauss model depending on $\ParV$, the functionals $h_{per}$ and $h_{iso}$ in (\ref{def-hper}), (\ref{def-hiso}) also satisfy \textbf{[H]}. In a similar way, test functions like $e^{\PrSc{\ParV}{\Vect{V}(x^m|\varphi)}}$, $e^{\PrSc{\ParV}{\Vect{V}(x^m|\varphi)}/2}$, $|\varphi_{\mathcal{B}(x,r)}|e^{\PrSc{\ParV}{\Vect{V}(x^m|\varphi)}}$, $\mathbf{1}_{[0,r]}(d(x^m,\varphi)) e^{\PrSc{\ParV}{\Vect{V}(x^m|\varphi)}}$ satisfy $\widetilde{\mbox{\textbf{[H]}}}$.

We  show that most of all the assumptions required in Theorems~\ref{prop-cons} and~\ref{prop-norm} are not too restrictive.

\begin{Proposition}\label{prop-ass}

$(i)$ Under Assumption \textbf{[Exp]}, Assumption \textbf{[Mod]} is fulfilled.\\
$(ii)$ For a test function satisfying \textbf{[H]} (resp. $\widetilde{\mbox{\textbf{[H]}}}$) and a model satisfying \textbf{[Exp]}  (resp. $\widetilde{\mbox{\textbf{[Exp]}}}$), then Assumptions \textbf{[C]} (excepted \textbf{[C3]}) and \textbf{[N]} are fulfilled.
\end{Proposition}

\begin{proof}
$(i)$ From  \textbf{[Exp]}, it is easy to check that  for any fixed $\ParV$, there exists a positive constant $K(\ParV)$ such that   $V(0^m|\varphi;\ParV)\geq -K(\ParV)$. Therefore the local energy is stable for any $\theta$. On the other hand it is finite range. These two properties imply the existence of ergodic measures (see {\it e.g.} \cite{A-BerBilDro99}, Proposition 1). Among them, there exists at least one stationary
measure due to the invariance by translation of the family of energy functions.

$(ii)$ \textbf{[C2]} and \textbf{[N3]} are quite obvious to check. Now, from the stability of the local energy, we have that for all $(m,\varphi)\in \mSp\times \Omega$, $\PrSc{\ParV} {\Vect{V}(0^m|\varphi)}\geq -\rho$ for $\rho<+\infty$, independent of $m,\varphi,\ParV$ (this is possible because $\SpPar$ is compact). Let us also underline that this property ensures that for every $\Lambda\Subset \RR[d]$, every $c\in\RR[]$,  $\Esp (e^{c|\Phi_\Lambda|} ) <+ \infty$ (see {\it e.g.} Proposition~11 of \cite{A-BerBilDro08}), which obviously implies that $\Esp (|\Phi_\Lambda|^\alpha ) <+\infty$ and $\Esp (|\Phi_\Lambda|^\alpha e^{c|\Phi_\Lambda|} )<+\infty$ for every $\alpha>0$. Now, under \textbf{[H]} and \textbf{[Exp]} (or $\widetilde{\mbox{\textbf{[H]}}}$ and $\widetilde{\mbox{\textbf{[Exp]}}}$), the expectations in \textbf{[C1]}, \textbf{[C4]}, \textbf{[N1]} and \textbf{[N4]} are clearly finite. Let us focus on \textbf{[C1]} for example (the justification for the other assumptions is similar). We have for any $\ParV,{\ParV}^\prime \in \SpPar$
$$
\Esp \left( |h(0^M,\Phi;\ParV)| e^{-\PrSc{{{\ParV}^\prime}}{\Vect{V}(0^M|\Phi)}} \right)\hspace*{10cm} $$
$$\hspace*{1.5cm} \leq \left\{
\begin{array}{ll}
e^\rho \Esp\left( |h(0^M,\Phi;\ParV)|\right) \leq c \times e^\rho \Esp \left( |\Phi_{\mathcal{B}(0,D)}|^\alpha \right) & \mbox{under \textbf{[H]} and \textbf{[Exp]}} \\
e^\rho \Esp\left( |\widetilde{h}(0^M,\Phi;\ParV)|e^{\PrSc{{{\ParV}}}{\Vect{V}(0^M|\Phi)}}  \right) \leq c \times e^\rho \Esp\left( |\Phi_{\mathcal{B}(0,D)}|^\alpha e^{c|\Phi_{\mathcal{B}(0,D)}|}\right) &\mbox{ under }\widetilde{\mbox{\textbf{[H]}}} \mbox{ and }\widetilde{\mbox{\textbf{[Exp]}}},
\end{array} \right.
$$
for some constants $\alpha$ and $c$. Note that, if we had not assumed $\widetilde{\mbox{\textbf{[Exp]}}}$ for test functions of the form $\widetilde{\mbox{\textbf{[H]}}}$, then one would have had expectations of the form $\Esp\left( e^{|\Phi_\Lambda|^k} \right)$ for some $k>1$ which is not necessarily finite under the local stability property. Assumption \textbf{[N2]} is proved similarly and by using the dominated convergence theorem.
\end{proof}

\begin{remark}
By following ideas in \cite{A-CoeDro09}, it is possible to fulfill the integrability type assumptions for more complicated models such as the Lennard-Jones model (which is not locally stable and nonlinear in terms of the parameters). The using of Ruelle's estimates \cite{A-Rue70} plays a crucial role in this case of superstable interaction. For the sake of conciseness and simplicity, we do not investigate this in the present paper. 
\end{remark}

%%%%%%%%%%%%%%%%%%%%%%%%%%%%%%%%%%%%%%%%%%%%%%%%%%%%%%%%%%%%%%%%%%%%%%%%%%%%%
%%%%%%%%%%%%%%%%%%%%%%%% Identifiabilité %%%%%%%%%%%%%%%%%%%%%%%%%%%%%%%%%%%%
%%%%%%%%%%%%%%%%%%%%%%%%%%%%%%%%%%%%%%%%%%%%%%%%%%%%%%%%%%%%%%%%%%%%%%%%%%%%%

\section{Identifiability : Assumption \textbf{[C3]}}\label{SectionIdent}

The Assumption \textbf{[C3]}  is related to the identifiability of the estimation procedure. It is more complicated to verify than the other assumptions and an investigation to obtain a criterion or a characterization seems necessary. We address this question in this section.

In the following, we consider that the interaction has an exponential form as in \eqref{formeexp}.
Then \textbf{[C3]} is equivalent to: $\ParV=\ParVT$ is the unique solution of the nonlinear system of equations in $\ParV$ defined by 
\begin{equation}\label{ident}
 \Esp \left( h_k(0^M,\Phi;\ParV) \left(
e^{-\PrSc{\ParV}{\Vect V(0^M|\Phi)}} -e^{-\PrSc{\ParVT}{\Vect V(0^M|\Phi)}}
\right) \right) = 0, \qquad 1\le k \le K.
\end{equation}
If $h_k$ and $\Vect V$ are sufficiently regular, each equation in (\ref{ident}) gives a $(p-1)$-dimensional manifold of solutions in $\SpPar$ containing $\ParVT$. So it is clear that the choice $K\ge p$ is in general necessary to prove that the system (\ref{ident}) admits the unique solution $\ParVT$. 

In Section \ref{pK}, we investigate the delicate case $K=p$ in detail. In opposition to the linear case where $p$ hyperplanes in $\RR[p]$ have in general a unique common point, the intersection of $p$ $(p-1)$-dimensional manifolds does not generally reduce to a single point. So, when $K=p$, there is no guarantee that   (\ref{ident}) has a unique solution $\ParVT$. This is illustrated by a simple example at the beginning of Section \ref{pK}.  In Proposition $\ref{propident}$, we provide a criterion to ensure that the system in (\ref{ident}) admits the only one solution $\ParVT$. Some examples, for which the criterion is available, are presented and the rigidity of the criterion, when $p\ge 3$, is also evoked. In the case where $p=2$, we show that our criterion is not far from being necessary.

% Moreover, in general, the choice $K=p$ is not sufficient. Indeed, in opposition to the linear case where $p$ hyperplanes in $\RR[p]$ have, in general, an unique common point, the intersection of $p$ $(p-1)$-dimensional manifolds is classically not reduced to a single point. In our opinion, this situation is not exotic in the setting of solutions in (\ref{ident}). Let us give a simple example, in the case $K=p=2$, where this situation occurs.

The case $K>p$ is studied in Section \ref{Kp}. The identification problem should be simpler since, in general, $p+1\;$ $(p-1)$-dimensional manifolds in $\RR[p]$ have no common point. We give a sufficient criterion to prove the identification but we think that it is far from being necessary.

% In the case $K=p$, assumptions \textbf{[Det]} and \textbf{[Id]} in Proposition \ref{propident} below ensure the dispersion around $\ParVT$ of manifolds of solutions coming from (\ref{ident}). Roughly speaking, it means that the distance, between solutions of each equation in (\ref{ident}), is increasing with respect to their distance with $\ParVT$. So our infinitesimal criterion implies that $\ParVT$ is the unique solution of system (\ref{ident}).

% In the Case $K>p$, the identification problem should be simple since, in general, $p+1$ $(p-1)$-dimensional manifolds in $\RR[p]$ have no common point. Obviously this property is not true in all situations but the cases, where it does not occur, should be exceptional. For example, if $p=2$ and $K=3$, three arbitrary curves in $\RR[2]$, having a common point $\ParVT$, may intersect each other several times but in general these intersection points involve only both of these three curves. Therefore $\ParVT$ should be the unique intersection point of the three curves.

% In Sections \ref{pK} and \ref{Kp} we investigate the cases $K=p$ and $K>p$ for an energy function $V$ having an exponential form (\ref{}). We present criterion to prove \textbf{[C3]} and comment their powerful. Proposition \ref{rigid} shows that the case $p=2$ is particularly interesting. Section \ref{p2} is devoted at this setting.

Before presenting these two sections, let us give further notation. We denote by $P_V$ the law of $\Vect V(0^M,\Phi)$ in $\RR[p]$. We also define  the function $\Psi_\ParV$, for each $\ParV\in\SpPar$, by 
\begin{equation}\label{Psi}
\begin{array}{cccc}
\Psi_\ParV: & \RR[p] & \longrightarrow  & \RR[K]\\
& \Vect{v} & \longmapsto & 
 \left( 
 \begin{array}{c}
 \Esp\Big(h_1(0^M,\Phi;\ParV) \Big| \Vect V(0^M|\Phi)=\Vect{v}\Big)\\
 \vdots\\
 \Esp\Big(h_K(0^M,\Phi;\ParV) \Big| \Vect V(0^M|\Phi)=\Vect{v} \Big)
 \end{array}\right)
\end{array}.
\end{equation}
We will see that this function plays a crucial role in the identification problem.

%  Let us remark that if $h_i=\psi_i(U)$ then the $i$th coordinate function of $\Psi_\ParV$ is nothing else than $\psi_i$. If $\Psi_\ParV$ can not be computed explicitly from $U$ and $(h_i)$, we think it is possible to approximate it from the data via for example kernel function estimators (Fred peux tu pr�ciser cela comme tu me l'avais expliqu�!). If functions $(h_i)$ do not depend on $\ParV$ we write $\Psi$ in place of $\Psi_\ParV$.

\subsection{The case $K=p$}\label{pK}

First of all, let us give a simple example to show that the identification problem is delicate in the situation where $K=p$.
Let us consider that $K=p=2$, $V_1=1$, and let us choose the simple test functions $h_1=1$ and $h_2=e^{\PrSc{\ParV}{\Vect V(x^m|\varphi)}}$. Then $\tilde\ParV=(\tilde\ParV_1,\tilde\ParV_2)$ with $\tilde\ParV_1=\ParVT_1-\ln(\Esp(e^{-\ParVT_2 V_2(0^M|\Phi)}))$ and $\tilde\ParV_2=0$ is always a solution of the system in (\ref{ident}). Therefore if $\ParVT_2\neq 0$ and if $\tilde\ParV$ defined before is in $\SpPar$, then the system  in (\ref{ident}) admits at least two solutions.

In the following, we first give a sufficient criterion to prove the identifiability and propose some examples. Next, we show the rigidity of our criterion which seems constraining when $p\ge 3$. 

\subsubsection{Criterion for identifiability}

Assumption \Det gathers the two following assumptions:
\begin{itemize}
\item[{\Det[($\neq$)]}] For every $\ParV$ in $\SpPar$, $det(\Vect{v}_1,\ldots,\Vect{v}_p)det\big(\Psi_\ParV(\Vect{v}_1),\ldots,\Psi_\ParV(\Vect{v}_p)\big)$ is not $(P_V)^{\otimes p}$-a.s. identically null
\item[{\Det[($\geq$)]}] For every $\ParV$ in $\SpPar$, there exists $\epsilon=\pm 1$ such that for $(P_V)^{\otimes p}$-a.s. every $(\Vect{v}_1,\ldots,\Vect{v}_p)$ in $(\RR[p])^p$
$$ \epsilon \;det(\Vect{v}_1,\ldots,\Vect{v}_p)det\big(\Psi_\ParV(\Vect{v}_1),\ldots,\Psi_\ParV(\Vect{v}_p)\big)\ge 0.$$
\end{itemize}
When $\epsilon=1$ (respectively $\epsilon=-1$), \Det[($\geq$)] means that $\Psi_\ParV$ preserves the sign (respectively the opposite sign) of the determinant.

The criterion is the following.
\begin{Proposition}\label{propident}
If $K=p$ then Assumption \Det ensures that Assumption \textbf{[C3]} holds.
\end{Proposition}
\begin{proof}
Denoting by $\zeta$ the real function $x\mapsto  \ln\left(\frac{e^x-1}{x}\right)$ with the convention $\zeta(0)=0$, the equations (\ref{ident}) become 
\begin{equation}\label{systscalar}
 \PrSc{(\ParVT-\ParV)}{\Vect X_k(\ParV,\ParVT)} = 0, \qquad 1\le k \le p,
\end{equation}
where the vector $\Vect X_k(\ParV,\ParVT)$ is defined by
\begin{equation}\label{Xi}
 \Vect X_k(\ParV,\ParVT)=\Esp \left(  h_k(0^M,\Phi;\ParV) e^{-\PrSc{\ParVT}{\Vect V(0^M|\Phi)}}e^{\zeta\big(\PrSc{(\ParVT-\ParV)}{\Vect V(0^M|\Phi)}\big)} \Vect V(0^M|\Phi)\right).
\end{equation}
Therefore the system (\ref{ident}) admits the unique solution $\ParVT$ if the family of vectors $(\Vect X_k(\ParV,\ParVT))_{1\le k \le p}$ is independent in $\RR[p]$. Let us give a formula of the determinant of these vectors which shows that it is not null.

Conditioning by the law of $\Vect V(0^M|\Phi)$ and using the multi-linearity of the determinant, we obtain
\begin{eqnarray*}
& &det(\Vect X_1(\ParV,\ParVT),\dots, \Vect X_p(\ParV,\ParVT))\\
& = &\int\!\cdots\!\int det\left(\Esp\left[h_1(0^M,\Phi;\ParV) e^{-\PrSc{\ParVT}{\Vect v_1}}e^{\zeta\big(\PrSc{(\ParVT-\ParV)}{\Vect v_1}\big)} \Vect v_1 \Big|\Vect V(0^M|\Phi)=\Vect v_1\right],\dots,\right.\\
& &\left.\Esp\left[h_p(0^M,\Phi;\ParV) e^{-\PrSc{\ParVT}{\Vect v_p}}e^{\zeta\big(\PrSc{(\ParVT-\ParV)}{\Vect v_p}\big)} \Vect v_p \Big|\Vect V(0^M|\Phi)=\Vect v_p\right] \right) P_V(d\Vect v_1)\cdots P_V(d\Vect v_p)\\
& = & \int\!\cdots\!\int e^{\sum_{k=1}^p -\PrSc{\ParVT}{\Vect v_k}+\zeta\big(\PrSc{(\ParVT-\ParV)}{\Vect v_k}\big)} det\left(\Vect v_1,\dots,\Vect v_p\right)\prod_{k=1}^p \Esp\left[h_k(0^M,\Phi;\ParV)|\Vect v_k\right]P_V(d\Vect v_1)\cdots P_V(d\Vect v_p)\\
&=&\frac{1}{p!}\sum_{\sigma\in S_p} \int\!\cdots\!\int e^{\sum_{k=1}^p -\PrSc{\ParVT}{\Vect v_{\sigma(k)}}+\zeta\big(\PrSc{(\ParVT-\ParV)}{\Vect v_{\sigma(k)}}\big)} det\left(\Vect v_{\sigma(1)},\dots,\Vect v_{\sigma(p)}\right) \\
& & \prod_{k=1}^p \Esp\left[h_k(0^M,\Phi;\ParV)|\Vect v_{\sigma(k)}\right] P_V(d\Vect v_1)\cdots P_V(d\Vect v_p),
\end{eqnarray*}
where $S_p$ is the set of all permutations in $\{1,\dots,p\}$. Denoting by $\epsilon(\sigma)$ the signature of $\sigma$, we obtain
\begin{eqnarray}\label{formDet}
& & det(\Vect X_1(\ParV,\ParVT),\dots, \Vect X_p(\ParV,\ParVT))\\
& = &\frac{1}{p!} \int\!\cdots\!\int e^{\sum_{k=1}^p -\PrSc{\ParVT}{\Vect v_k}+\zeta\big(\PrSc{(\ParVT-\ParV)}{\Vect v_k}\big)} det\left(\Vect v_1,\dots,\Vect v_p\right) \sum_{\sigma\in S_p} \epsilon(\sigma)\prod_{k=1}^p \Esp\left[h_k(0^M,\Phi;\ParV)|\Vect v_{\sigma(k)}\right] \nonumber\\
& & P_V(d\Vect v_1)\cdots P_V(d\Vect v_p)\nonumber\\
&= &\frac{1}{p!} \int\!\cdots\!\int e^{\sum_{k=1}^p -\PrSc{\ParVT}{\Vect v_k}+\zeta\big(\PrSc{(\ParVT-\ParV)}{\Vect v_k}\big)}det(\Vect{v}_1,\ldots,\Vect{v}_p)det\big(\Psi_\ParV(\Vect{v}_1),\ldots,\Psi_\ParV(\Vect{v}_p)\big) P_V(d\Vect v_1)\cdots P_V(d\Vect v_p).\nonumber
\end{eqnarray}
From Assumption \Det, this determinant is not null. The proposition is proved.
\end{proof}
\bigskip

Now let us give some examples for which the criterion is valid. 

\begin{example}[linear case] If the function $\Psi_\ParV$ is linear and invertible then Assumption \Det[($\geq$)] is clearly satisfied and \Det[($\neq$)] holds as soon as the support of $P_V$ is not included in a hyperplane. In particular, if $h_k=V_k$ for every $1\le k \le p$ then $\Psi_\ParV$ is equal to the identity function. This situation corresponds to the pseudo-likelihood procedure for which we regain the identifiability via our criterion.
\end{example}

\begin{example}[area process] \label{exArea} For the area process defined in Section~\ref{sec-quermass} with $\lambda^{\mE}=\delta_R$ (i.e. the radii of balls are constant), it is easy to check that the functions $h_{per}$ and $h_{iso}$ respectively defined by~\eqref{def-hper} and~\eqref{def-hiso} give a function $\Psi$ which satisfies Assumption \Det. Indeed the support of $P_V$ is the segment $\{1\}\times[0,\pi R^2]$ in $\RR[2]$ and for a vector $\Vect v=(1,v_2)$ the image $\Psi(\Vect v)$ is $(\psi_1(v_2),0)$ if $v_2\neq \pi R^2$ and $(2\pi R,1)$ if $v_2= \pi R^2$. Therefore, it follows that \Det[($\geq$)] is satisfied and noting that $0<P_V((1,\pi R^2))<1$ we deduce that \Det[($\neq$)] holds too. 
\end{example}

\begin{example}[a general example with $\pmb{p=2}$]
Example \ref{exArea} is included in a more general setting when $p=2$. Indeed let us suppose that the function $\Psi_\ParV$ has the form $\Psi_\ParV( v_1, v_2) = \big(g_\ParV(v_1, v_2),g_\ParV(v_1, v_2)f_\ParV(v_2/v_1)\big)$ where $g_\ParV$ is a nonnegative scalar function and $f_\ParV$ is a monotone scalar function. Then $\Psi_\ParV$ satisfies \Det[($\geq$)] and \Det[($\neq$)] holds if $g_\ParV( v_1, v_2)f_\ParV( v_2/ v_1)$ is not $P_V^{\otimes 2}$-a.s. constant when $g_\ParV( v_1, v_2)$ is not null.
\end{example}

\begin{example}[functions of the type $\pmb{h_k e^{\PrSc{\ParV}{\Vect V}} }$] Let us suppose that the functions $(h_k)$ ensure that $\Psi_\ParV$ satisfies \Det[($\geq$)] then for any nonnegative function  $g_\ParV$ from $\RR[p]$ to $\mathbb{R}$ the functions $(\tilde h_k)=\big(g_\ParV(\Vect V)h_k\big)$ also provide a function $\tilde\Psi_\ParV$ satisfying \Det[($\geq$)]. This remark is related to Section \ref{sec-quick} where it is suggested to choose functions $(\tilde h_k)$ of the form $(e^{\PrSc{\ParV}{\Vect V}} h_k)$ to simplify the integral in (\ref{eq-defC}).\\
As an immediate consequence, the test functions $\big(V_k e^{\PrSc{\ParV}{\Vect{V}}}\big)$, considered in \cite{A-Billiot97} for the particular multi-Strauss point process, satisfy \Det[($\geq$)].
\end{example}

\subsubsection{Rigidity of the criterion}

In this section, we give some comments about the rigidity of the criterion. In  Proposition \ref{rigid} below, we show that a function $\Psi_\ParV$, satisfying \Det[($\geq$)], has a strong linear structure since, under very reasonable assumptions, the image of any hyperplane is included in a hyperplane. For example, in the classical setting where  $V_1=1$, the function $\Psi_\ParV$ is defined from the affine space  $H=\{1\}\times\RR[p-1]$ and if $\Psi_\ParV$ is assumed to be continuous then the image of any $p-2$ dimensional affine space in $H$ is included in a hyperplane. This property clearly shows that $\Psi_\ParV$ is very rigid when $p\ge 3$.

However, when $p=2$, we show in Proposition \ref{propnonident} that our criterion is not far from being necessary. Indeed, we present a large class of examples which do not satisfy our criterion and for which the identifiability fails.

\begin{Proposition}\label{rigid}
Let $\Psi$ be a continuous function from $\mathcal{D}$ to $\RR[p]$ satisfying \Det[($\geq$)], where the domain $\mathcal{D}$ is a subset of $\RR[p]$ with the following property: for any $(x_i)_{1\le i\le p} \in \mathcal{D}^p$ such that $det(x_1,\dots,x_p)=0$, then for any neighborhood $\mathcal{V}$ of $(x_i)$, there exist $(x^+_i)$ and $(x^-_i)$ in $\mathcal{V}\cap\mathcal{D}^p$ such that $det(x^+_1,\dots,x^+_p)>0$ and $det(x^-_1,\dots,x^-_p)<0$.
Then for any hyperplane $H$ in $\RR[p]$ the image $\Psi(H\cap \mathcal{D})$ is included in a hyperplane of $\RR[p]$.
\end{Proposition}

\begin{proof}
Let $H$ be a hyperplane in $\RR[p]$. To prove that $\Psi(H\cap \mathcal{D})$ is included in a hyperplane, it is sufficient to prove that the dimension of the vectorial space generated by the vectors in $\Psi(H\cap \mathcal{D})$ is not equal to $p$. Let us suppose that it is equal to $p$, then there exists $(x_i)_{1\le i\le p}$ in $(H\cap\mathcal{D})^p$
such that $det(\Psi(x_1),\dots,\Psi(x_p))\neq 0$. Since $dim(H)=p-1$, we have $det(x_1,\dots,x_p)=0$. By continuity of $\Psi$ and by the local properties of $\mathcal{D}$ assumed in Proposition \ref{rigid}, we find $(x^+_i)_{1\le i\le p}$ and $(x^-_i)_{1\le i\le p}$ in $\mathcal{D}^p$ such that $det(x^+_1,\dots,x^+_p)det(\Psi(x^+_1),\dots,\Psi(x^+_p))>0$ and $det(x^-_1,\dots,x^-_p)det(\Psi(x^-_1),\dots,\Psi(x^-_p))<0$, which contradicts Assumption \Det[($\geq$)].
\end{proof}
\bigskip

In the case where $\mathcal{D}=\RR[p]$, if we assume that $\Psi(\RR[p])$ is not reduced to a hyperplane and that $\Psi$ is differentiable at the origin, then we can show that $\Psi$ satisfies \Det[($\geq$)] if and only if $\Psi(x)=g(x)Ax$, where $A$ is an invertible matrix and $g$ a nonnegative  scalar function. It means that $\Psi$ is quasi linear and so the rigidity of $\Psi$ is very strong.

Now let us focus on the case where $p=2$ and let us show that, while our criterion seems very constraining, it is not far from being necessary in this case. We suppose that $p=2$, $V_1=1$ and that the support of $V_2$ is included in an interval $[a,b]$. Let us remark that this case occurs for the area process with $[a,b]=[0,\pi R^2]$. First of all, it is easy to check visually, depending on the geometry of $\gamma_\ParV$ defined by the curve $\Psi_\ParV(\{1\}\times[a,b])$, whether $\Psi_\ParV$ satisfies \Det (see figure \ref{graphPsi} for examples).

Moreover let us show that the criterion is not far from being necessary. 
%Suppose that the functions $(h_i)$ do not depend on $\ParV$ and that $\Psi=\Psi_\ParV$ satisfies for $\epsilon=\pm 1$
%{\bf i)} There exists $\delta>0$ such that for $P_V^{\otimes 2}$-a.s. every $(\Vect v_1,\Vect v_2)$ in $(\{1\}\times[a,a+\delta])^2$  
%$$\epsilon\; det\left(\Vect v_1,\Vect v_2\right) \left(\Psi(\Vect v_1),\Psi(\Vect v_2)\right)\ge 0$$
%and $P_V^{\otimes 2}$-a.s. non identically null. Similarly we suppose 
%{\bf ii)} There exists $\delta>0$ such that for $P_V^{\otimes 2}$-a.s. every $(\Vect v_1,\Vect v_2)$ in $(\{1\}\times[b-\delta,b])^2$ 
%$$\epsilon\; det\left(\Vect v_1,\Vect v_2\right) \left(\Psi(\Vect v_1),\Psi(\Vect v_2)\right)\le 0$$
%and $P_V^{\otimes 2}$-a.s. non identically null.
Suppose that the functions $(h_i)$ do not depend on $\ParV$ and that $\Psi:=\Psi_\ParV$ satisfies for $\epsilon=\pm 1$ Assumption \DetT decomposed into the three following assumptions:
\begin{itemize}
\item[{\DetT[($\neq$)]}] $det\left(\Vect v_1,\Vect v_2\right) det\left(\Psi(\Vect v_1),\Psi(\Vect v_2)\right)$ is not $P_V^{\otimes 2}$-a.s. identically null
\item[{\DetT[($\geq$)]}] there exists $\delta>0$ such that for $P_V^{\otimes 2}$-a.s. every $(\Vect v_1,\Vect v_2)$ in $(\{1\}\times[a,a+\delta])^2$  
$$\epsilon\; det\left(\Vect v_1,\Vect v_2\right) det\left(\Psi(\Vect v_1),\Psi(\Vect v_2)\right)\ge 0$$
\item[{\DetT[($\leq$)]}] there exists $\delta>0$ such that for $P_V^{\otimes 2}$-a.s. every $(\Vect v_1,\Vect v_2)$ in $(\{1\}\times[b-\delta,b])^2$ 
$$\epsilon\; det\left(\Vect v_1,\Vect v_2\right) det\left(\Psi(\Vect v_1),\Psi(\Vect v_2)\right)\le 0.$$
\end{itemize}
See Figure \ref{graphPsi} for an example of such $\Psi$. Obviously, this situation is not exactly the opposite of Assumption {\bf [Det]}, but it is strongly related to it. Then, we have the following proposition which proves that the identifiability fails for this large class of examples. 

%{\bf iii)} The functions $h_1$ et $h_2$ are nonnegative.

\begin{Proposition}\label{propnonident}
If the functions $(h_i)$ are nonnegative, if $\Psi$ satisfies \DetT and  if $det\Big( E_{P_V}(\Psi(\Vect v)\tr{\Vect v })\Big)\neq 0$ then \textbf{[C3]} fails.
\end{Proposition}

Let us note that even if the assumption $det\Big( E_{P_V}(\Psi(\Vect v)\tr{\Vect v})\Big)\neq 0$ seems unnatural, it is in general satisfied.
 
\begin{proof}
Let us show that (\ref{ident}) admits another solution than $\ParVT$. We only give here the main lines of the proof. 

We denote by $\mathcal{O}$ the set in $\RR[2]$ containing the vectors $\Vect u$ which are orthogonal to at least one vector $\Vect v$ in $\{1\}\times[a,b]$, i.e. $\mathcal{O}=\{\Vect u\in \RR[2], \exists \Vect v\in \{1\}\times[a,b], \PrSc{\Vect u}{\Vect v}=0\}$. In fact, $\mathcal O$ is the union of a cone $\mathcal O^+$ in the upper half plane  and a cone $\mathcal O^-$ in the lower half plane. For any $\delta>0$, the expression of the determinant in (\ref{formDet}) can be split in two parts
\begin{eqnarray}\label{detcoupe}
& & det(\Vect X_1(\ParV,\ParVT), \Vect X_2(\ParV,\ParVT))\\
&= &\frac{1}{2} \int\int_{[a,a+\delta]^2} e^{\sum_{k=1}^2 -\PrSc{\ParVT}{\Vect v_k}+\zeta\big(\PrSc{(\ParVT-\ParV)}{\Vect v_k}\big)}det(\Vect{v}_1,\Vect{v}_2)det\big(\Psi(\Vect{v}_1),\Psi(\Vect{v}_2)\big) P_V(d\Vect v_1) P_V(d\Vect v_2)\nonumber\\
& & +\frac{1}{2} \int\int_{[a,b]^2\backslash [a,a+\delta]^2} e^{\sum_{k=1}^2 -\PrSc{\ParVT}{\Vect v_k}+\zeta\big(\PrSc{(\ParVT-\ParV)}{\Vect v_k}\big)}det(\Vect{v}_1,\Vect{v}_2)det\big(\Psi(\Vect{v}_1),\Psi(\Vect{v}_2)\big) P_V(d\Vect v_1) P_V(d\Vect v_2).\nonumber
\end{eqnarray} 
Let $\Vect u\neq 0$ in $\mathcal O^+$ and $\ParV=\ParVT+\alpha \Vect u$ with $\alpha>0$. From \DetT[($\neq$)] and \DetT[($\geq$)] since $\zeta$ is increasing and $\zeta(x)\sim x$ as $x\rightarrow +\infty$, we deduce that the first integral in (\ref{detcoupe}) dominates the second one when $\alpha$ goes to infinity. Therefore $ det(\Vect X_1(\ParV,\ParVT), \Vect X_2(\ParV,\ParVT))$ has the  sign $\epsilon$ (defined before \DetT) when $\alpha$ is large enough. Similarly, if $\Vect u$ is in $\mathcal O^-$ then from \DetT[($\neq$)] and \DetT[($\leq$)] $ det(\Vect X_1(\ParV,\ParVT), \Vect X_2(\ParV,\ParVT))$ has the  sign  $-\epsilon$ when $\alpha$ is large enough and it implies that the sign of $det(\Vect X_1(\ParV,\ParVT),\Vect X_2(\ParV,\ParVT))$, with $\ParV=\ParVT+ \Vect u$, is different for $\Vect u$ in $\mathcal O^+$ or in $\mathcal O^-$ as soon as $\vert \Vect u \vert $ is large enough.  We deduce that there exists a continuous curve $t \mapsto  \Vect u(t)$ which crosses $\mathcal O$ such that  $det(\Vect X_1(\ParV(t),\ParVT),\Vect X_2(\ParV(t),\ParVT))=0$ for every $\ParV(t)=\ParVT+ \Vect u(t)$. Let us note that the assumption $det\Big( E_{P_V}(\Psi(\Vect v)\tr{\Vect v})\Big)\neq 0$ ensures that $det(\Vect X_1(\ParVT,\ParVT),\Vect X_2(\ParVT,\ParVT))\neq 0$, and so $\Vect u(t)$ is never null. Let us show that there exists $t_0$ such that $\ParV(t_0)$ is a solution of the system  in~(\ref{ident}). 
 
Since the functions $(h_i)$ are nonnegative and from Definition (\ref{Xi}), we obtain that for every $t$, $\Vect X_1 (\ParV(t),\ParVT)$ is collinear to a vector in $\{1\}\times[a,b]$. By continuity of the function $t \mapsto  \PrSc{\Vect X_1 (\ParV(t),\ParVT)}{\Vect u(t)}$ and by the mean value theorem, there exists $t_0$ such that $\Vect u(t_0)$ is orthogonal to $\Vect X_1 (\ParV(t_0),\ParVT )$. Since the determinant $det(\Vect X_1(\ParV(t_0),\ParVT),\Vect X_2(\ParV(t_0),\ParVT))=0$, 
$\Vect u(t_0)$ is also orthogonal to $\Vect X_2 (\ParV(t_0),\ParVT )$ and it follows that the system in~(\ref{systscalar}) provides at least two solutions $\ParVT$ and $\ParV(t_0)$. Identification assumption (\ref{ident}) or \textbf{[C3]} fail. 
\end{proof}

%Chez moi je suis oblig� d'enlever le suffixe en png pour que ca marche!!!!! d�sol�!

\begin{figure}[htbp]
  \begin{center}
  \begin{picture}(450,150) 
 \put(0,5){\includegraphics[scale=0.25]{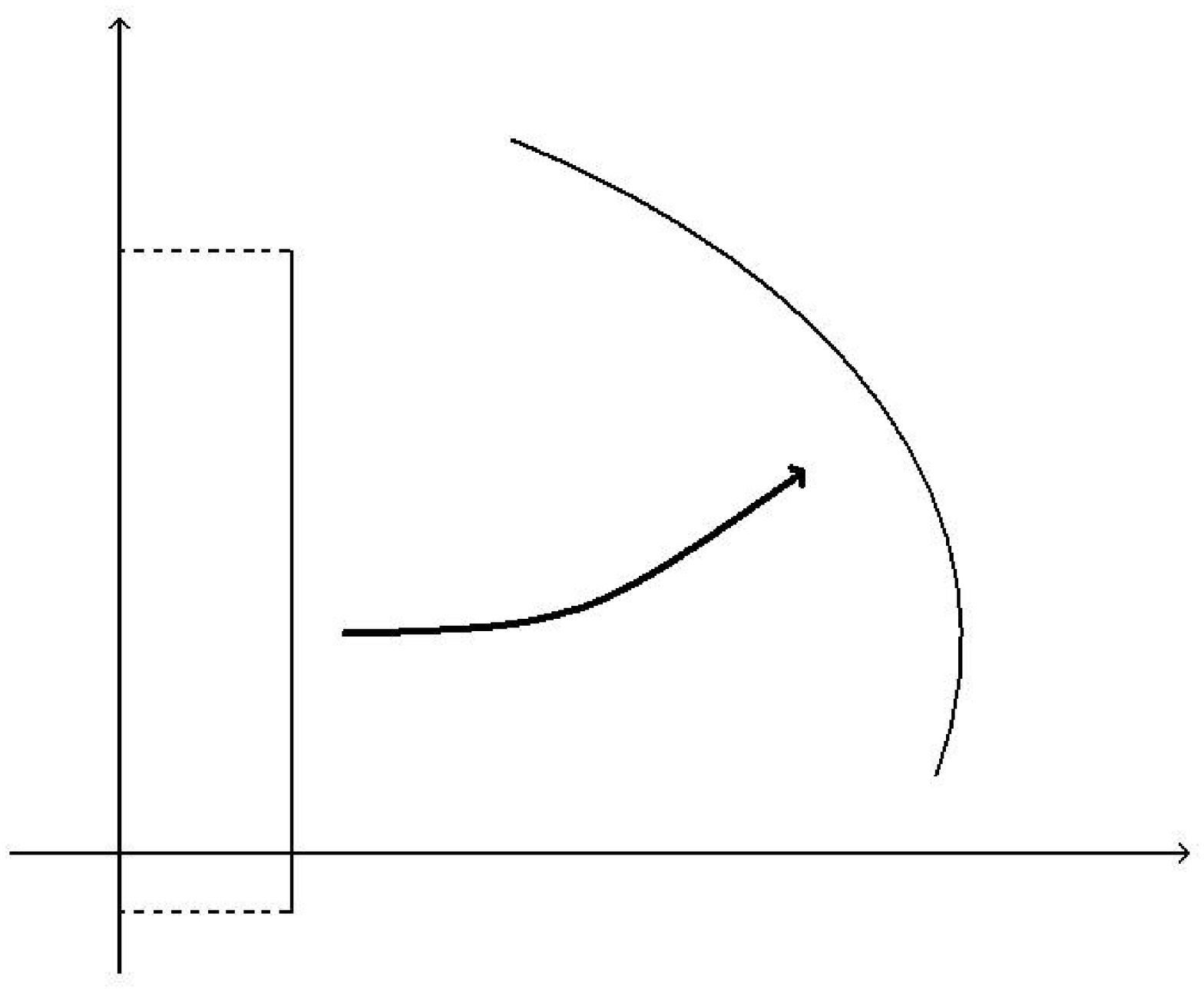}}
 \put(150,5){\includegraphics[scale=0.25]{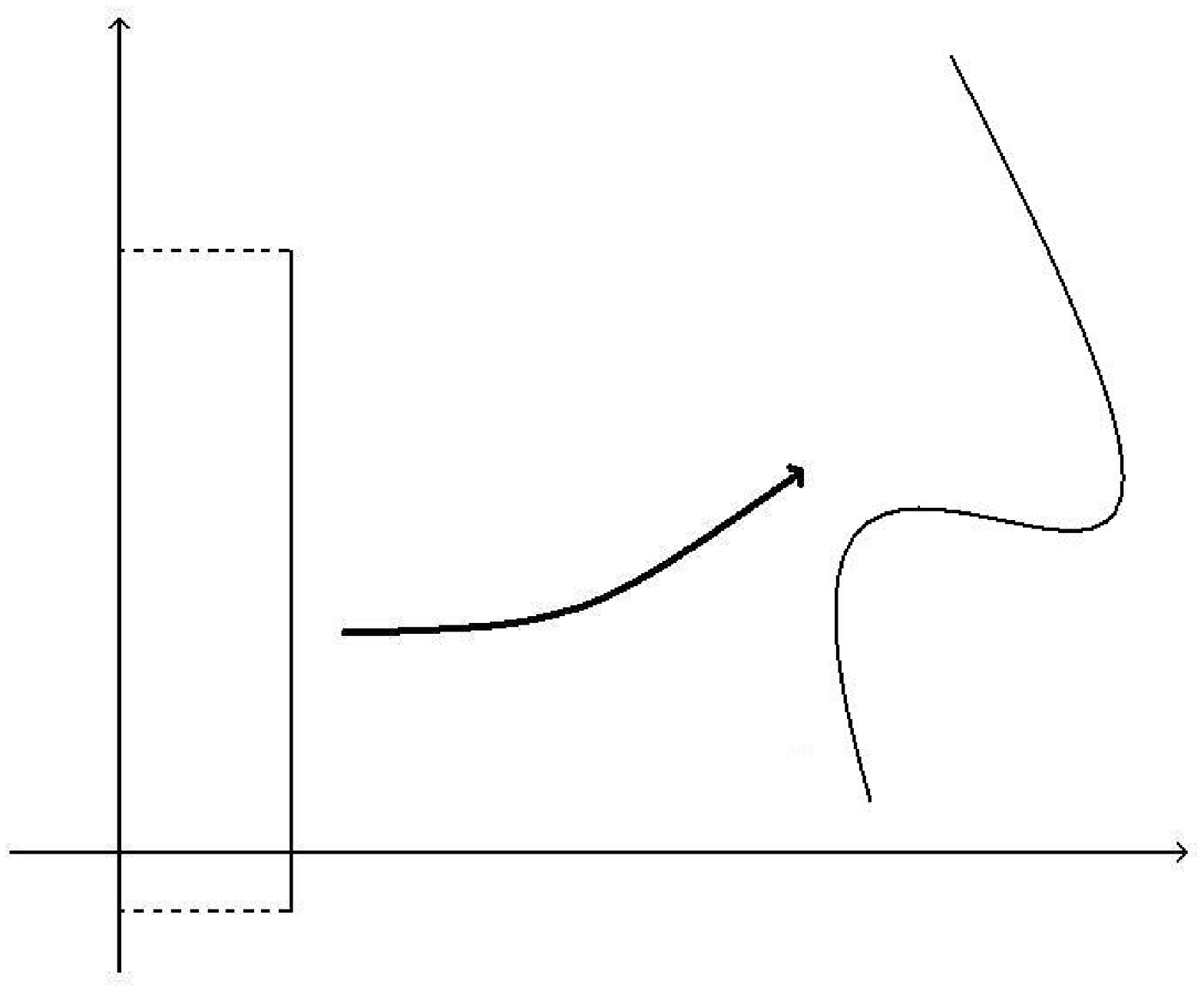}}
 \put(300,5){\includegraphics[scale=0.25]{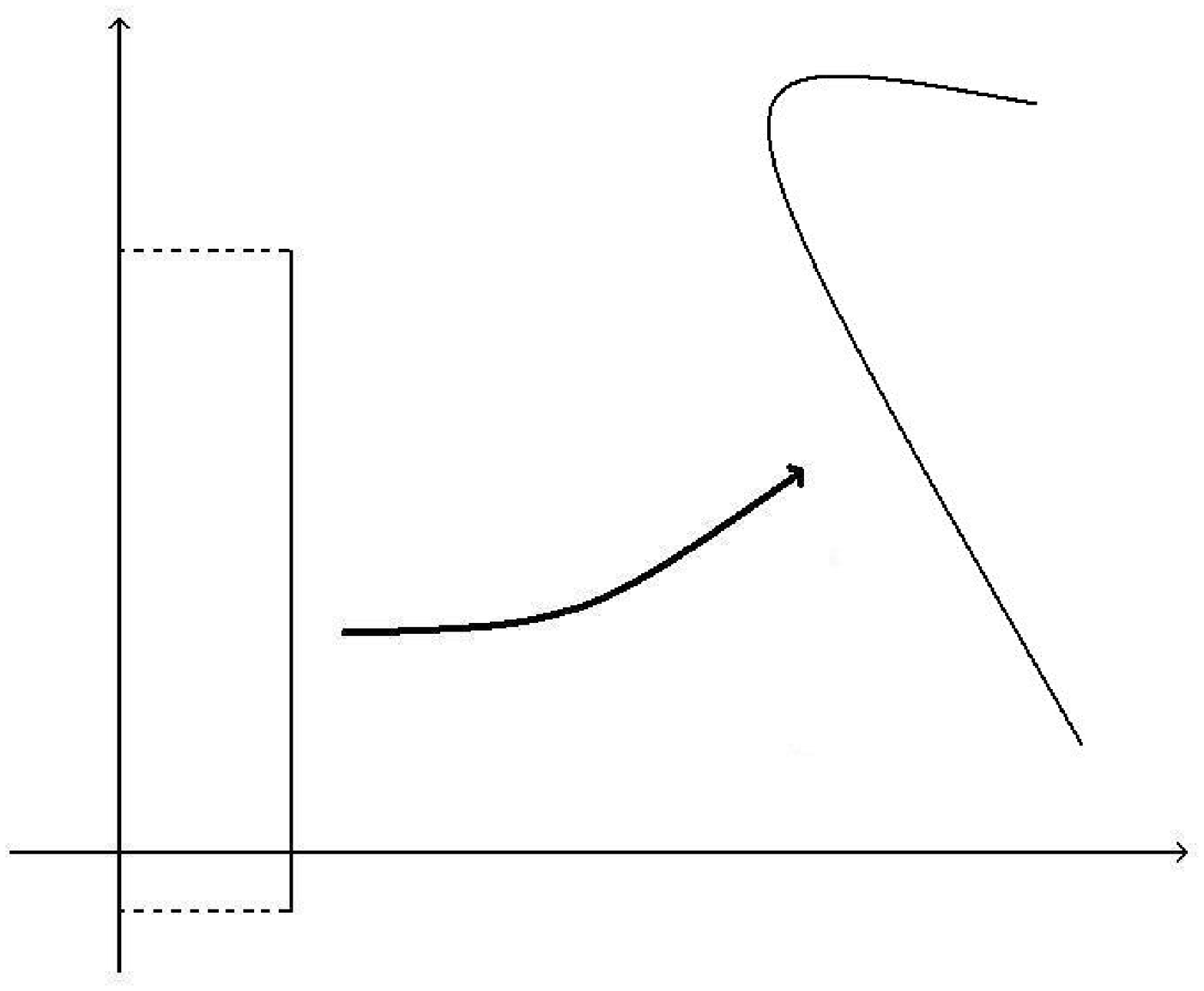}}
  \put(7,12){$a$}
  \put(7,85){$b$}
  \put(40,12){$1$}
  \put(55,60){$\Psi$}
  \put(105,75){$\gamma$}
  \put(157,12){$a$}
  \put(157,85){$b$}
  \put(190,12){$1$}
  \put(205,60){$\Psi$}
  \put(255,75){$\gamma$}
  \put(307,12){$a$}
  \put(307,85){$b$}
  \put(340,12){$1$}
  \put(355,60){$\Psi$}
  \put(405,75){$\gamma$}
  
  \end{picture}

  \caption{{\small on the left (respectively in the middle), an example of function $\Psi$ satisfying (respectively not satisfying) \Det[($\geq$)]. On the right, an example of $\Psi$ satisfying \DetT. }}\label{graphPsi}
  \end{center}
  \end{figure}

% We will present and study in details the case $p=2$ in Section \ref{p2}. We will see that our criterion seems really useful and that it is not far from being sufficient (see Proposition \ref{propnonident}).  

% \begin{Proposition}
% Let $\Psi$ a continuous function from $\RR[p-1]$ to $\RR[p]$ which is differentiable at $0$. For every $\Vect {u}$ in $\RR[p-1]$ we define $\Vect {\tilde u}=(1,\Vect {u})$ in $\RR[p]$ and $\tilde \Psi(\Vect {\tilde u})= \Psi(\Vect u)$. Then $\tilde \psi$ satisfies \textbf{[Det]} for every $\Vect {\tilde u_1},\dots ,\Vect {\tilde u_P}$ in $(\{1\}\times\RR[p-1])^p$ if and only if
% \begin{itemize}
% \item in the case $p=2$, $\Psi_2/\Psi_1$ is monotone where $\Psi=(\Psi_1,\Psi_2)$.
% \item in the case $p\ge 3$, there exist functions $g$ and $\rho$ from $\RR[p-1]$ to $\mathbb{R}$, where $g$ has a constant sign, and a invertible linear function $A$ in $\RR[p]$ such that $\Psi(\Vect u)=g(\Vect u)A(\rho(\Vect u),\Vect u)$.
% \end{itemize} 
% \end{Proposition}
% \begin{proof}

\subsection{The case K>p}\label{Kp}

 In the case where $K>p$, we noticed, in the introduction, that the identification problem should be simpler. Nevertheless, we did not find a satisfactory criterion to prove it. The following Proposition \ref{propidentb} gives a sufficient criterion which is probably far from being necessary. It is based on a slight modification of Assumption \Det which does not seem to be the appropriate tool in this setting. However, in the case where $p=2$ and $K=3$, this condition  reduces to a nice geometrical property which can be checked easily. 

First, let us present the criterion. We denote by $\mathcal A$ the set of all subsets with $p$ elements in $\{1,\ldots,K\}$,
$$ \mathcal{A}= \Big\{ I\subset \{1,\ldots,K\}, \text{ such that } \#(I)=p\Big\}.$$
We say that Assumption \Det['] is satisfied if, for every $\ParV$ in $\SpPar$, there exists a family of real coefficients $(c_I)_{I\in\mathcal{A}}$ such that the two following assumptions hold:
\begin{itemize}
\item[{\Det['($\neq$)]}] $\displaystyle\sum_{I\in\mathcal{A}} c_I det(\Vect{v}_1,\dots,\Vect v_p)det\big(\Psi^I_\ParV(\Vect{v}_1),\dots,\Psi^I_\ParV(\Vect{v}_p))$ is not $(P_V)^{\otimes p}$-a.s. identically null
\item[{\Det['($\geq$)]}]  for $(P_V)^{\otimes p}$-a.s. every $(\Vect{v}_1,\ldots,\Vect{v}_p)$ in $(\RR[p])^p$
$$ \sum_{I\in\mathcal{A}} c_I det(\Vect{v}_1,\dots,\Vect v_p)det\big(\Psi^I_\ParV(\Vect{v}_1),\dots,\Psi^I_\ParV(\Vect{v}_p))\ge 0,$$
\end{itemize}
where $\Psi^I_\ParV(\Vect v)$ denotes the $p$-dimensional vector extracted from $\Psi_\ParV(\Vect v)$ for the coordinates given by $I$.
In the particular case where $K=p$, Assumption \Det['] becomes Assumption \Det exactly.

Our criterion is the following
\begin{Proposition}\label{propidentb}
Assumption \Det[']   ensures that Assumption \textbf{[C3]} holds.
\end{Proposition}
\begin{proof}
As in the proof of Proposition \ref{propident}, to show that (\ref{ident}) admits the unique solution $\ParVT$, it is sufficient to prove that there exists $I$ in $\mathcal{A}$ such that, for all $\ParV\not=\ParVT$, $det\big((\Vect X_i(\ParV,\ParVT))_{i\in I}\big)\neq 0$. It is equivalent to: for every $\ParV\not=\ParVT$ in $\SpPar$ there exists a family of real coefficients $(c_I)_{I\in\mathcal{A}}$ such that
\begin{equation}\label{sum}
  \sum_{I\in\mathcal{A}} c_I det\big((\Vect X_i(\ParV,\ParVT))_{i\in I}\big) >0.
\end{equation}
With calculations as in (\ref{formDet}) we obtain 
\begin{eqnarray*}
& & p!\sum_{I\in\mathcal{A}} c_I det\big((\Vect X_i(\ParV,\ParVT))_{i\in I}\big)\\
&= &\int e^{\sum_{k=1}^p -\PrSc{\ParVT}{\Vect v_k}+\zeta\big(\PrSc{(\ParVT-\ParV)}{\Vect v_k}\big)}\Bigg(\sum_{I\in\mathcal{A}} c_I det(\Vect{v}_1,\dots,\Vect v_p)det\big(\Psi^I_\ParV(\Vect{v}_1),\dots,\Psi^I_\ParV(\Vect{v}_p))\Bigg) P_V(d\Vect v_1)\cdots P_V(d\Vect v_p).\nonumber
\end{eqnarray*}
Thanks to \Det['], this quantity is positive.
\end{proof}

% Proposition \ref{propidentb} is based on a criterion of the dispersion of solutions in the non linear system (\ref{ident}). If $K>p$, we noticed in the introduction of Section \ref{SectionIdent} that the system (\ref{ident}) may have only one solution even if any sub system with $p$ equations provides several solutions. So a dispersion criterion seems not very  well-adapted in this setting and we think that Propostion \ref{propidentb} provides a sufficient condition to obtain identifiability \textbf{[C3]} which is far from being necessary. Nevertheless, in the case $p=2$ presented in the next section, we will see that if $K=3$ our criterion provides a nice geometrical condition. 

In the case where $p=2$ and $K=3$,  \Det['($\geq$)] is satisfied if and only if there exist $a,b,c$ in $\RR[3]$ such that for every $\Vect v_1$ and $\Vect v_2$ with $det(\Vect v_1,\Vect v_2)>0$
\begin{equation}\label{dd}
 a \;det \left(\Psi^{\{1,2\}}_\ParV(\Vect v_1),\Psi^{\{1,2\}}_\ParV(\Vect v_2)\right) +
b \;det \left(\Psi^{\{1,3\}}_\ParV(\Vect v_1),\Psi^{\{1,3\}}_\ParV(\Vect v_2)\right)+
c \;det \left(\Psi^{\{2,3\}}_\ParV(\Vect v_1),\Psi^{\{2,3\}}_\ParV(\Vect v_2)\right)\geq 0.
\end{equation}
If we denote by $\wedge$ the vectorial product in $\RR[3]$, the inequality in~(\ref{dd}) means that the following set
\begin{equation}\label{vectoriel}
\Big\{ \Psi_\ParV(\Vect v_1)\wedge \Psi_\ParV(\Vect v_2), \text{ for all } \Vect v_1,\Vect v_2 \text{ such that } det(\Vect v_1,\Vect v_2)>0  \Big\}
\end{equation}
is included in the half space in $\RR[3]$ with equation $cx-by+az\ge 0$. In the setting where $V_1=1$ and $V_2$ is included in an interval $[a,b]$, as for the area process,  this condition becomes a geometrical characteristic of the curve $\gamma_\ParV=\Psi_\ParV(\{1\}\times[a,b])$ in $\RR[3]$, which is easy to check visually. 

%  Therefore  This geometrical condition can be easily checked  if  $\Psi$ is explicit or if $\Psi$ is numerically approximated from the data. 

%%%%%%%%%%%%%%%%%%%%%%%%%%%%%%%%%%%%%%%%%%%%%%%%%%%%%%%%%%%%%%%%%%%%%%%%%%%%%%
%%%%%%%%%%%%%%%%%%%%%%%%%%%% Cas non h�r�ditaire %%%%%%%%%%%%%%%%%%%%%%%%%%%%%
%%%%%%%%%%%%%%%%%%%%%%%%%%%%%%%%%%%%%%%%%%%%%%%%%%%%%%%%%%%%%%%%%%%%%%%%%%%%%%

\section{Extension in the presence of non-hereditary interaction} \label{non-hereditary}

In several recent papers, Gibbs processes with non hereditary interactions are considered, in particular in the domain of stochastic geometry (see \cite{A-Der05}, \cite{A-DerDroGeo09}). The parametric estimation of such models has also been investigated. The first results in this direction have been given in \cite{A-DerLav09} via a pseudo-likelihood procedure based on a generalization of the Georgii-Nguyen-Zessin formula (\ref{GNZnonstat}). The same kind of generalization is  possible for the Takacs-Fiksel procedure. We address this improvement in this section. 

In the following, we do not assume that the energy $V_\Lambda(\varphi,\ParV)$ satisfies the heredity assumption 
(\ref{heredite}). The first consequences are that the local energy $V(x^m|\varphi;\ParV)$ is not defined in general and that the Georgii-Nguyen-Zessin  formula  is not available. Let us begin by presenting the generalization of this formula,  as stated in \cite{A-DerLav09}, Proposition 2, which is valid in the hereditary and non-hereditary settings.

 We first need to recall the concept of removable points which has been introduced in \cite{A-DerLav09}, Definition~3.

\begin{definition}\label{removable} A point $x^m$ in a configuration $\varphi$ is called removable if there exists a bounded set $\Lambda$ containing $x$ such that $V_\Lambda(\varphi\backslash x^m,\ParV)<+\infty$. We denote by $\mathcal{R}_{\ParV} (\varphi)$ the set of removable points in $\varphi$.
\end{definition}

Let us remark that the removable set is only related to the support of the underlying Gibbs measure. The local energy $V(x^m|\varphi\backslash x^m;\ParV)$ of any removable point $x^m\in\mathcal{R}_{\ParV}(\varphi)$ can then be defined by the classical expression (\ref{energielocale}) where $\Lambda$ comes from Definition~\ref{removable}. In the hereditary case, all the points of $\varphi$ are removable and we regain the classical definition of the local energy. 

The generalization of the Georgii-Nguyen-Zessin formula is the following equation
\begin{equation}\label{GNZnonstatnonH}
\Esp\left( \ism[ \mathbb{R}^d \times \mSp]{h\left(x^m,\Phi;\ParV\right) e^{- \VIPar{x^m}{\Phi}{\ParVT}}} \right) = 
\Esp\left( \sum_{x^m \in \mathcal{R}_{\ParVT}(\Phi) } h\left(x^m,\Phi\setminus x^m;\ParV\right)  \right) .
\end{equation}
Let us notice that the only  difference with the classical formula is that the sum is restricted to the removable points. Now, let us present the consequences of this formula on the Takacs-Fiksel procedure. We have to consider the two following cases: 

\begin{itemize}
\item When the support of the Gibbs measure does not depend on  $\ParV$:   the set of removable points $\mathcal{R}_{\ParV} (\varphi)$ does not depend on  $\ParV$ either. In this case, the Takacs-Fiksel estimor is defined by (\ref{eq-defTakacs}), and $C_\Lambda$ is as in (\ref{eq-defC}) to the exception of the  sum which is restricted to the removable points:

\begin{equation} \label{eq-defCnonH}
C_{\Lambda}(\varphi;h,\ParV) := \ism[\Lambda \times \mSp]{ h(x^m,\varphi;\ParV) e^{-V(x^m|\varphi;\ParV)} } -\sum_{x^m\in \mathcal{R}_{\ParVT}(\varphi)\cap \Lambda} h(x^m,\varphi\setminus x^m;\ParV).
\end{equation}

The sum is computable because by assumption, the set $\mathcal{R}_{\ParVT} (\varphi)$ does not depend on $\ParVT$. In this situation, with the same assumptions  \textbf{[C]} and \textbf{[N]}, the consistency and the asymptotic normality of the estimator may be proved as in Section 7. 

\item When the support of the Gibbs measure depends on some parameters  $\ParV_{hc}=(\ParV_1,\ldots,\ParV_q)$, $q\leq p$ (called the hardcore parameters): the remaining  parameters $\ParV_{sm}=(\ParV_{q+1},\ldots,\ParV_p)$ are supposed to parameterize the classical (or smooth) interaction between points. The set of removable points $\mathcal{R}_{\ParV} (\varphi)$ therefore depends on $\ParV_{hc}$ only. The estimation issue is more complicated in this case. Indeed, Assumption \textbf{[C2]} requires some regularities of the interaction with respect to the parameter $\ParV$, such as continuity, which clearly fail to be true for  the support parameter $\ParV_{hc}$. The Takacs-Fiksel procedure is therefore unable to estimate ${\ParV}_{hc}$. Note that this problem is not specific to the presence of non-hereditary interactions, but arises as soon as some hardcore parameters have to be estimated. In \cite{A-DerLav09}, the authors solve this problem in both the hereditary and non-hereditary setting, by introducing a two-step estimation procedure. We can follow the same strategy here. In a first step, the estimator ${\hat\ParV}_{hc}$ of the hardcore parameter is defined in a natural way according to the observed support of the point process (see Section 4.2.1 in \cite{A-DerLav09}). Then, in a second step, the Takacs-Fiksel estimator $\hat\ParV_{sm}$ is defined by (\ref{eq-defTakacs}) with

\begin{eqnarray} 
C_{\Lambda}(\varphi;h,\ParV_{sm}) &:=& \ism[\Lambda \times \mSp]{ h(x^m,\varphi;\ParV_{sm},\hat\ParV_{hc}) e^{-V(x^m|\varphi;\ParV_{sm},\hat\ParV_{hc})} } \nonumber \\
&& \qquad\qquad -\sum_{x^m\in \mathcal{R}_{\hat \ParV_{hc}}(\varphi)\cap \Lambda} h(x^m,\varphi\setminus x^m;\ParV_{sm},\hat\ParV_{hc}).\label{eq-defCnonH2}
\end{eqnarray}

Let us remark that the estimator ${\hat\ParV}_{hc}$ is plugged in the computation of $C_\Lambda$. In particular, the removable points are determined with respect to ${\hat\ParV}_{hc}$. As in \cite{A-DerLav09}, the regularity and integrability assumptions of type \textbf{[C]} for $\hat\ParV_{sm}$ and conditions on the support of the Gibbs measure are required in order to obtain the consistency of $(\hat\ParV_{hc},\hat\ParV_{sm})$.  The asymptotic normality is more difficult to obtain and no general results are available. In fact, it seems that there is no hope to expect asymptotic normality without managing the rate of convergence of $\hat\ParV_{hc}$, which should strongly depend on the model.
\end{itemize}

%%%%%%%%%%%%%%%%%%%%%%%%%%%%%%%%%%%%%%%%%%%%%%%%%%%%%%%%%%%%%%%%%%%%%%%%%%%%%%
%%%%%%%%%%%%%%%%%%%%%%%%%%% Simulation %%%%%%%%%%%%%%%%%%%%%%%%%%%%%%%%%%%%%%%
%%%%%%%%%%%%%%%%%%%%%%%%%%%%%%%%%%%%%%%%%%%%%%%%%%%%%%%%%%%%%%%%%%%%%%%%%%%%%%

\section{A short simulation study} \label{sec-simuls}

The first aim of the present paper is to explore the asymptotic properties of the Takacs-Fiksel estimate. Given a model, the important question of correctly choosing the test functions is an open question and will not be treated here. Also, a complete comparison between all the existing parametric estimation methods has not been attempted. In this section, we just present a brief simulation study to illustrate the methodology and the asymptotic results.

The model considered in this section is the (non-marked) Strauss process described by (\ref{strauss}) with two parameters. Table~\ref{tab-sim} and Figure \ref{fig-boxplot} give an empirical comparison of the approximated maximum likelihood estimate (MLE) described in \cite{A-Huang99}, the maximum pseudo-likelihood estimate (MPLE), which is a particular case of the Takacs-Fiksel estimate, and the explicit Takacs-Fiksel estimate (TFE) described in Section~\ref{sec-strauss} and in particular in \eqref{TFEexplicit}. To generate Strauss point processes and to compute the MLE, the \texttt{R} package \texttt{spatstat} is used. The MPLE and the TFE are computed respectively from~\eqref{eq-defLPL} and~\eqref{eq-defC} 
where the integrals were approximated by Monte-Carlo.\footnote{The MPLE is also implemented in \texttt{spatstat} but we did not use it. In fact, it appears that the implementation therein leads to biased estimations.} Strauss point processes are generated on the window $[0,3]^2\oplus R$ where $R$ is set to $0.05$. Recall that this parameter corresponds to the finite range of the Strauss process. Then the estimates are computed on the window $[0,\tau]^2\oplus R$ for $\tau=1,2,3$, using a $R$-erosion. We use the parametrization of \texttt{spatstat}, i.e. the intensity parameter is $\beta^\star:=e^{-\theta_1^\star}$ and the interaction parameter is $\gamma^\star:=e^{-\theta_2^\star}$. We set $\beta^\star=100$ and $\gamma^\star=0.5$. The estimation of $(\beta^\star,\gamma^\star)$ by the TFE is thus $(\widehat \beta,\widehat \gamma)=(e^{-\widehat{\ParV}_{1}},e^{-\widehat{\ParV}_{2}})$ where  $(\widehat{\ParV}_{1},\widehat{\ParV}_{2})$ are defined in  (\ref{TFEexplicit}).

As expected, the three estimates have decreasing variance when the domain of observation grows. %We notice a bias in the estimations derived by the pseudo-likelihood method that we do not explain\footnote{This bias seems to come from a coding error in the MPLE approximation in \texttt{spatstat} (personal communication from A. Baddeley).}.  
Since the explicit TFE is obtained very quickly, its behavior appears very satisfactory in comparison with the MLE and MPLE.

\begin{table}[htbp]
\begin{center}
\begin{tabular}{rrrr}
\hline
\multicolumn{4}{c}{Estimates of the parameter $\beta^\star=100$} \\
$W_\tau$& MLE& MPLE & TFE \\
\hline
$[0,1]^2$ & 103.50 (18.24) &101.19 (16.92)&101.44 (18.95) \\
$[0,2]^2$ & 101.03 (8.04) & 100.33 (8.66)& 99.99 (8.71)  \\
$[0,3]^2$ & 100.38   (6.58) & 99.55 (5.47) & 99.96 (6.27)  \\
\hline
\end{tabular}

\vspace*{1cm}

\begin{tabular}{rrrr}
\hline
\multicolumn{4}{c}{Estimates of the parameter $\gamma^\star=0.5$}\\
$W_\tau$ & MLE& MPLE & TFE \\
\hline
$[0,1]^2$ & 0.50 (0.17)   & 0.50 (0.16)  &0.52 (0.22) \\
$[0,2]^2$ &  0.50 (0.08)  & 0.50 (0.08) & 0.51(0.11) \\
$[0,3]^2$ &  0.51 (0.05)  & 0.49 (0.05) & 0.51 (0.06)\\
\hline
\end{tabular}
\end{center}
\caption{Empirical means and standard deviations between brackets for the MLE, MPLE and the (explicit) TFE based on $m=500$ replications of a Strauss point process with parameters $\beta^\star=100, \gamma^\star=0.5$, and (known) interaction range $R=0.05$, generated in the window $[0,3]^2\oplus R$  and estimated on the window $W_\tau=[0,\tau]^2\oplus R$ with $R$-erosion for $\tau=1,2,3$.}
\label{tab-sim}
\end{table}

\begin{figure}[htbp]
\begin{center}
\begin{tabular}{ll}
\includegraphics[scale=.4]{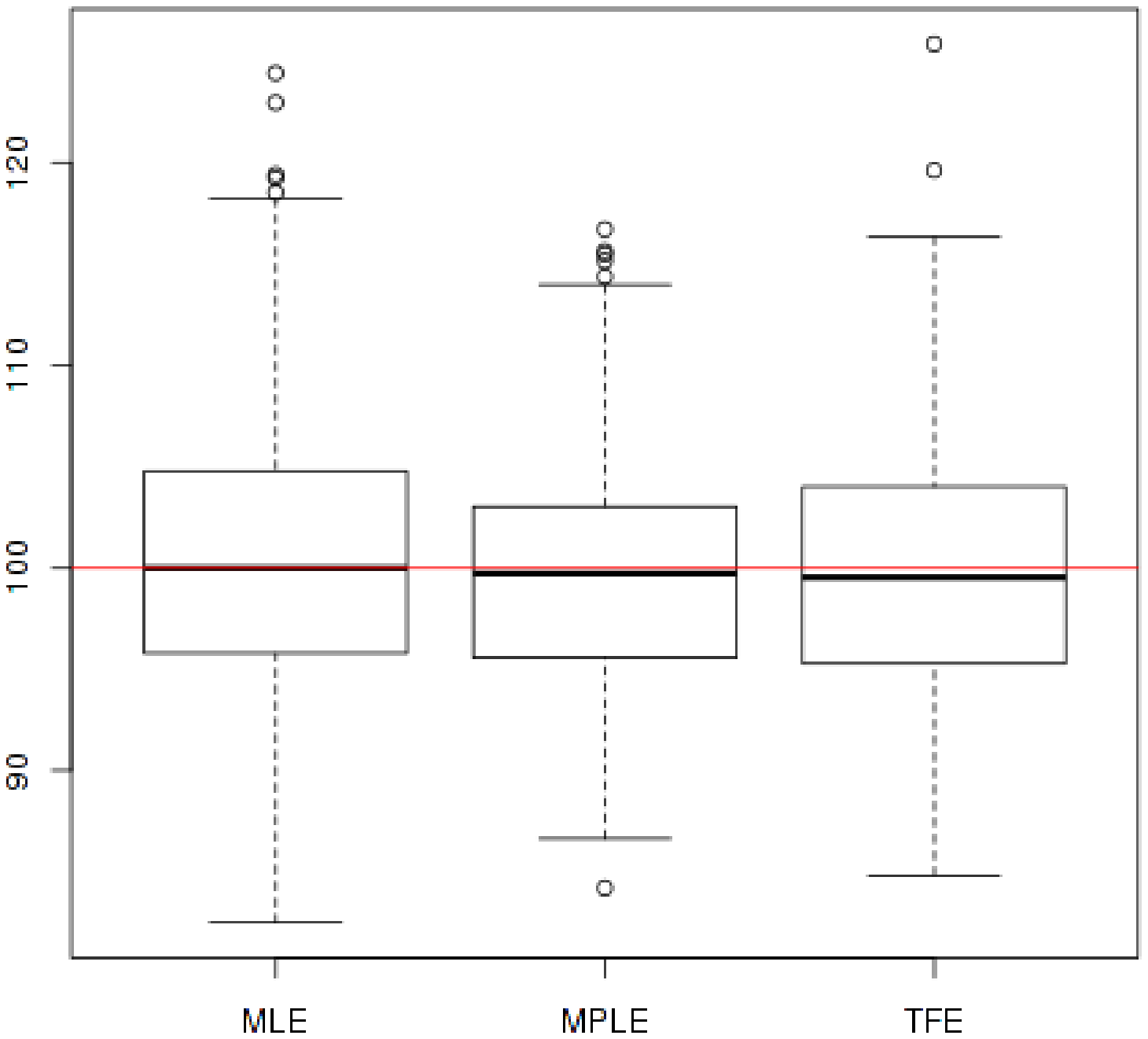} & \includegraphics[scale=.4]{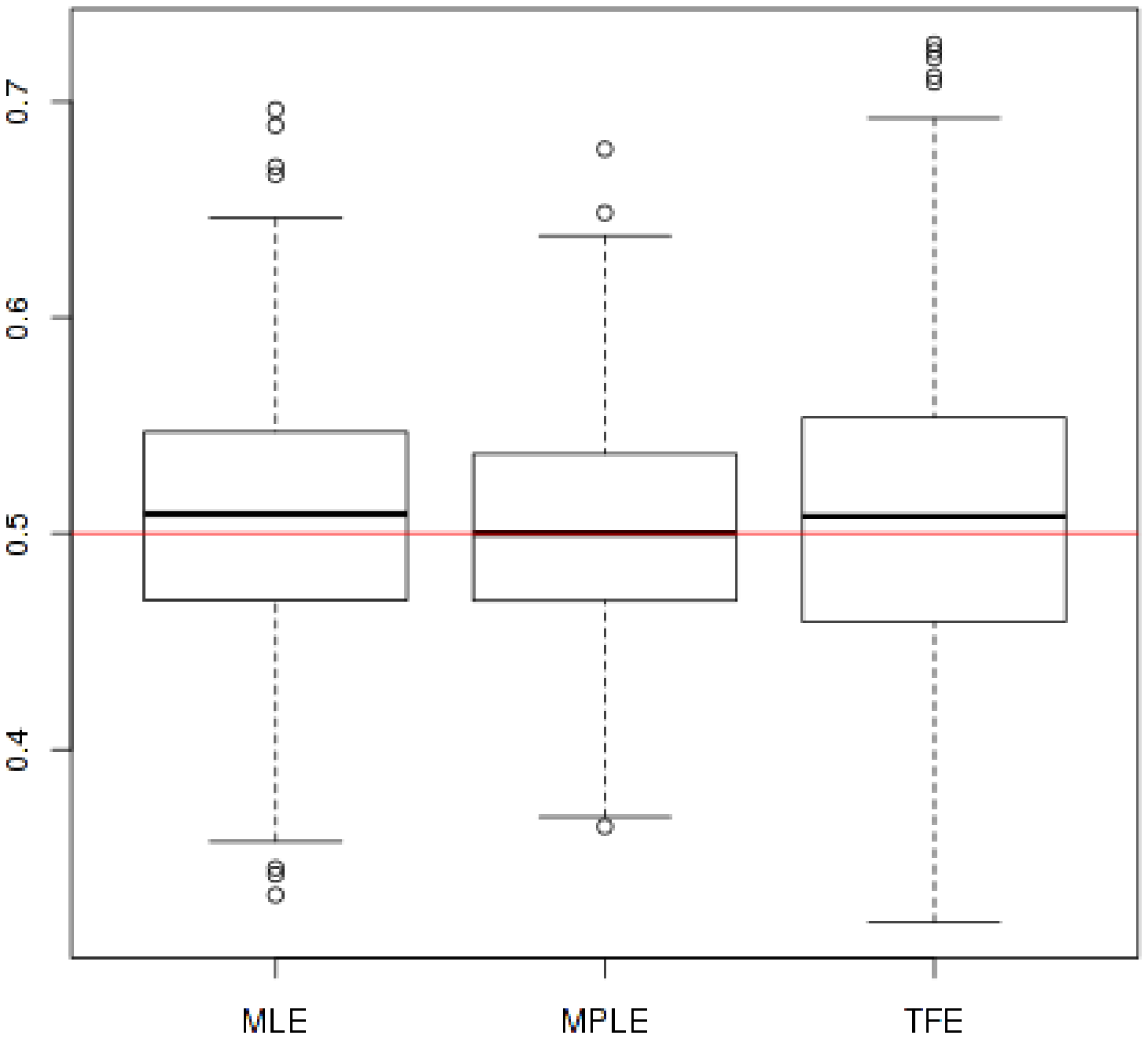}
\end{tabular}
\end{center}
\caption{Boxplots of the 3 estimates of $\beta^\star$ (left) and $\gamma^\star$ (right) when $\tau=3$, based on $m=500$ replications.}
\label{fig-boxplot}
\end{figure}

Now,  we focus more particularly on the explicit TFE given by \eqref{TFEexplicit}. We wish to assess whether the empirical covariance matrix of  $(\widehat{\theta_1},\widehat{\theta_2})^T$ agrees with the asymptotic covariance matrix deduced from \eqref{conv-norm} in Theorem \ref{prop-norm}. For this reason, we come back to the parameterization of the Strauss model of Section~\ref{sec-strauss}. %we compare the empirical covariance $\tau^2 \times \widehat{cov}(\widehat{\theta_1},\widehat{\theta_2})$ for $\tau=1,2,3$ to in Table~\ref{tab-est-cov}
From \eqref{conv-norm}, we deduce that the asymptotic covariance matrix of $\tau\times (\widehat{\theta_1},\widehat{\theta_2})^T$, as $\tau$ tends to $+\infty$, is ${\mathcal{E}}^{-1}(\Vect{h},\ParVT)^T \Mat{\Sigma}(\Vect{h},\ParVT) {\mathcal{E}}^{-1}(\Vect{h},\ParVT)$ (a simplification occurs here since $K=p=2$). Moreover, we can prove that the matrix $\mathcal{E}(\Vect{h},\ParVT)$ reduces to 
$$
\mathcal{E}(\Vect{h},\ParVT) = |\Delta|^{-1}\left( \begin{array}{ll}
\Esp\left(N_{0,\Delta} \right) & e^{\theta_2^\star} \Esp\left( N_{1,\Delta}\right) \\
0 & e^{\theta_2^\star} \Esp\left( N_{1,\Delta}\right)
\end{array}\right)
$$
where $\Delta$ is any bounded domain. Since the matrices $\mathcal{E}(\Vect{h},\ParVT)$ and $\Mat{\Sigma}(\Vect{h},\ParVT)$ depend on expectations with respect to $P_\ParVT$, they are approximated by Monte-Carlo simulations. Table~\ref{tab-est-cov}
contains the  empirical covariance of $\tau\times (\widehat{\theta_1},\widehat{\theta_2})^T$, for $\tau=1,2,3$ as well as the asymptotic covariance matrix approximated by Monte-Carlo simulations. 
Finally, in order to appreciate the Gaussian limit, Figure~\ref{fig-tcl} compares the empirical distribution of the two estimates  $(\widehat{\theta_1},\widehat{\theta_2})$ with the expected limit when $\tau=3$.
Note that the same kind of results as in Table~\ref{tab-est-cov} and Figure~\ref{fig-tcl} could have been easily obtained for $(\widehat \beta,\widehat \gamma)$ by use of the delta method.

\begin{table}[Htbp]
\begin{center}
\begin{tabular}{ccc|c}
\hline
\multicolumn{3}{c|}{Empirical covariance of $\tau \times (\widehat{\theta_1},\widehat{\theta_2})^T$} & MC approximation of \\
$\tau=1$ &$\tau=2$& $\tau=3$ &$ ({\mathcal{E}^T})^{-1} \Mat{\Sigma} \mathcal{E}^{-1}$ \\
 \hline &&&\\
$\left(\begin{array}{rr}
0.037 &-0.051\\
-0.051 & 0.185
\end{array}\right)$
&
$\left(\begin{array}{rr}
0.031 & -0.048\\
-0.048 & 0.185
\end{array}\right)$
&
$\left(\begin{array}{rr}
0.035 & -0.051 \\
 -0.051 & 0.179 
\end{array}\right)$ 
&
$\left(\begin{array}{rr}
0.034 &-0.045\\
-0.045 & 0.154 \\
\end{array}\right)$ \\
\hline
\end{tabular}
\end{center}
\caption{Comparison between the renormalized empirical covariance matrix of the estimates $(\widehat{\theta_1},\widehat{\theta_2})$ (based on $m=500$ replications of the Strauss process on $[0,\tau]^2\oplus R$ as in Table \ref{tab-sim}) and the asymptotic covariance matrix (approximated by Monte-Carlo simulations).}
\label{tab-est-cov}
\end{table}

\begin{figure}[htbp]
\begin{center}
\begin{tabular}{ll}
\includegraphics[scale=.4]{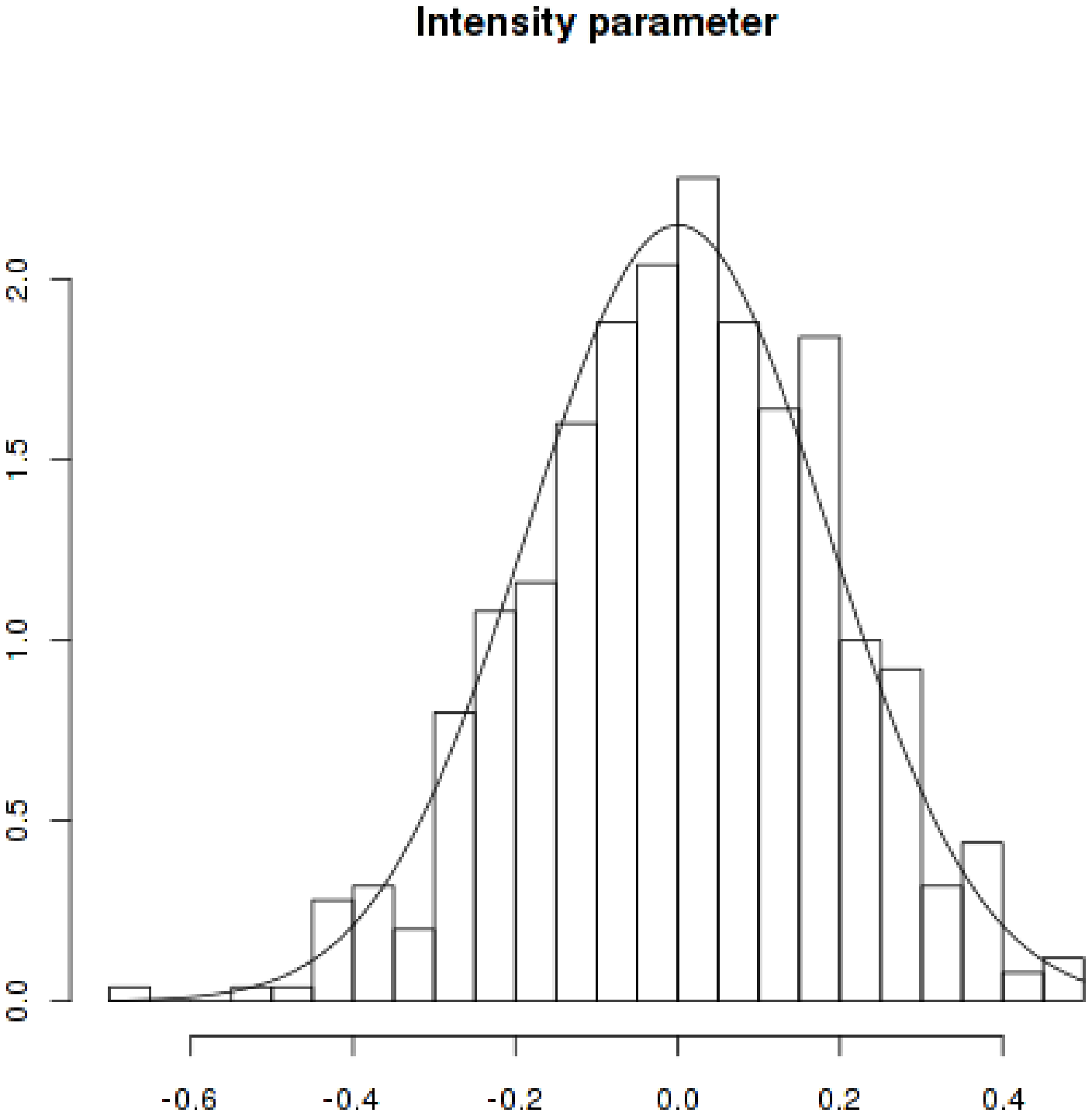} & \includegraphics[scale=.4]{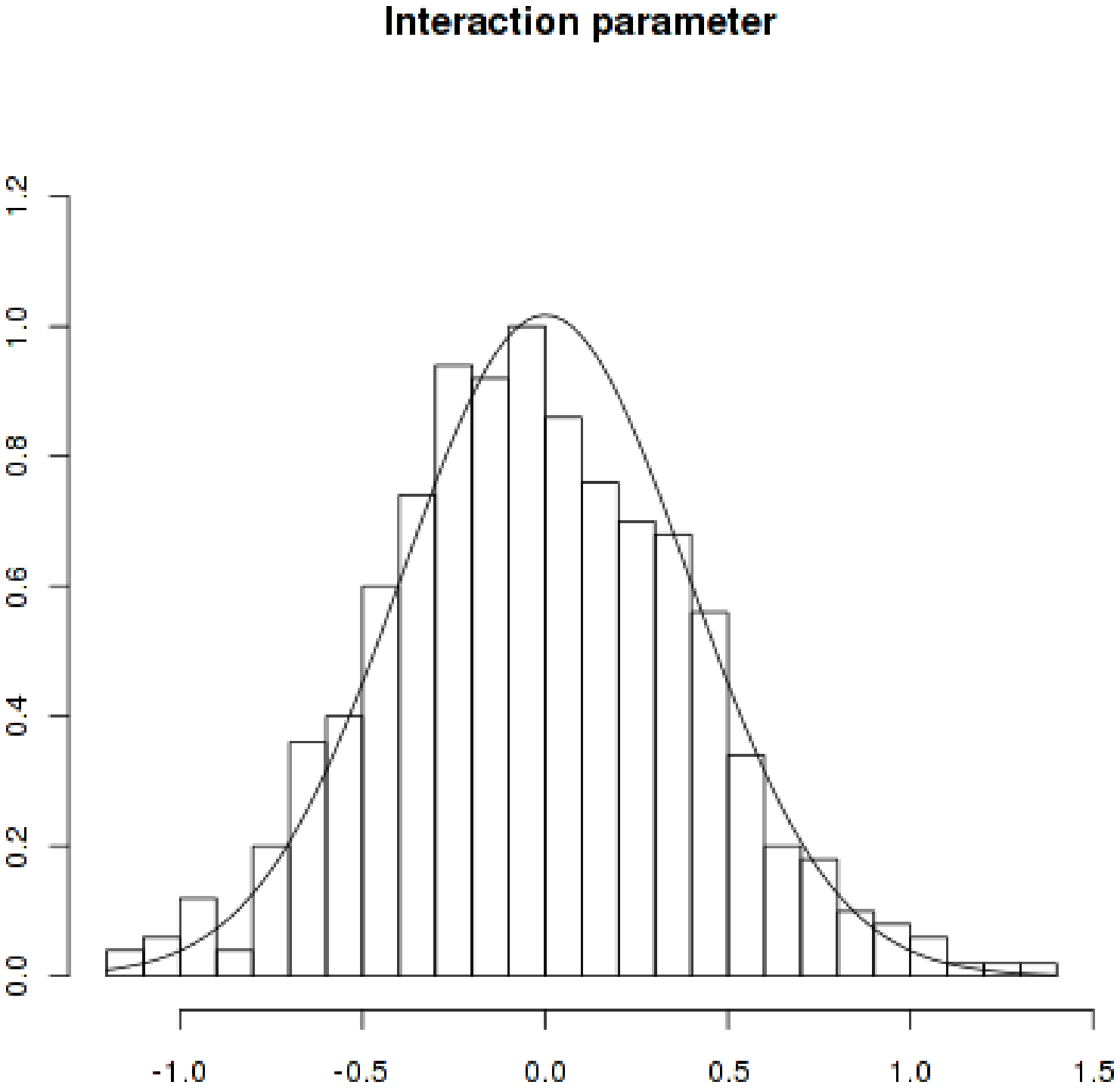}
\end{tabular}
\end{center}
\caption{Histogram of $\tau (\widehat{\theta}_j-\theta_j^\star)$ for $j=1$ (intensity parameter, left) and $j=2$ (interaction parameter, right), for the large window, i.e. $\tau=3$. The curves correspond to the densities of the asymptotic Gaussian laws, which are centered with variance .034 (left) and .154 (right) according to Table \ref{tab-est-cov}.}
\label{fig-tcl}
\end{figure}

%%%%%%%%%%%%%%%%%%%%%%%%%%%%%%%%%%%%%%%%%%%%%%%%%%%%%%%%%%%%%%%%%%%%%%%%%%%%%%
%%%%%%%%%%%%%%%%%%%%%%%%%%% Preuves %%%%%%%%%%%%%%%%%%%%%%%%%%%%%%%%%%%%%%%%%%
%%%%%%%%%%%%%%%%%%%%%%%%%%%%%%%%%%%%%%%%%%%%%%%%%%%%%%%%%%%%%%%%%%%%%%%%%%%%%%

\section{Proofs of asymptotic results} \label{sec-proofs}

\subsection{Proof of Theorem~\ref{prop-cons}}

If there is more than one stationary Gibbs measure, then some non ergodic Gibbs measures automatically exist because, in the convex set of all Gibbs measures, only the extremal measures are ergodic. But any stationary Gibbs measure can be represented as a mixture of ergodic measures (\cite{B-Geo88}, Theorem 14.10). Due to this decomposition, we can assume that $P_{\ParVT}$ is ergodic to prove the consistency of the Takacs-Fiksel estimate. The proof is split into two steps.

{\it Step 1.} $U_{\Lambda_n}$ is a contrast function.
Under \textbf{[C1]}, the ergodic Lemma \ref{ergodiclemma} and the GNZ formula given in~(\ref{GNZ}) may be applied to prove that $P_{\ParVT}$-a.s.
\begin{eqnarray*}
|\Lambda_n|^{-1} C_{\Lambda_n}(\Phi,h_k;\ParV) &\to& \Esp\left( h_k(0^M,\Phi;\ParV) e^{-V(0^M|\Phi;\ParV)} \right) - \Esp\left(h_k(0^M,\Phi\setminus 0^M;\ParV)  \right) \\
&=& \Esp\left( h_k(0^M,\Phi;\ParV) \left( e^{-V(0^M|\Phi;\ParV)}  - e^{-V(0^M|\Phi;\ParVT)} \right)
\right).
\end{eqnarray*}
Therefore, as $n\to +\infty$, one obtains $P_{\ParVT}-$a.s.
$$
U_{\Lambda_n}(\Phi;\Vect{h},\ParV) \to U(\ParV) :=\sum_{k=1}^K \Esp\left( h_k(0^M,\Phi;\ParV) \left( e^{-V(0^M|\Phi;\ParV)}  - e^{-V(0^M|\Phi;\ParVT)} \right)
\right)^2.
$$
Note that $U(\ParVT)=0$. In addition with Assumptions \textbf{[C2]} and \textbf{[C3]}, this proves that $U_{\Lambda_n}$ is a continuous contrast function vanishing only at $\ParVT$.

{\it Step 2.} Modulus of continuity

The modulus of continuity of the contrast process is defined for all $\varphi\in \Omega$ and all $\eta>0$ by 
$$
W_n(\varphi,\eta) = \sup \left\{ 
\Big|\ensuremath{U_{\Lambda_n}(\varphi;\Vect{h},\ParV)} - \ensuremath{U_{\Lambda_n}(\varphi;\Vect{h},\ParV^\prime)} \Big|: \ParV,\ParV^\prime \in \SpPar, || \ParV - \ParV^\prime || \leq \eta
\right\}.
$$
This step aims at proving that there exists a sequence $(\varepsilon_\ell)_{\ell \geq 1}$, with $\varepsilon_\ell \to 0$
as $\ell \to +\infty$ such that for all $\ell \geq 1$
\begin{equation} \label{modCont}
P \left( \limsup_{n \to +\infty}  \left( 
W_n \left(\Phi,\frac1\ell\right) \geq \varepsilon_\ell
\right)\right) = 0.
\end{equation}
Let $\ParV,\ParV^\prime \in \SpPar$, then 
\begin{eqnarray}
|U_{\Lambda_n}(\varphi;\Vect{h},\ParV)-U_{\Lambda_n}(\varphi;\Vect{h},\ParV^\prime)| &\leq& 
\sum_{k=1}^K \bigg\{ |\Lambda_n|^{-1} |C_{\Lambda_n}(\varphi;h_k,\ParV)-C_{\Lambda_n}(\varphi;h_k,\ParV^\prime)| \times \nonumber \\
&& \qquad  |\Lambda_n|^{-1} \big(|C_{\Lambda_n}(\varphi;h_k,\ParV)| + |C_{\Lambda_n}(\varphi;h_k,\ParV^\prime)| \big) \bigg\}. \label{mod1}
\end{eqnarray}
Under Assumption \textbf{[C4]}, there exists $n_0\geq 1$ such that for all $n\geq n_0$ we have for $P_{\ParVT}-$a.e. $\varphi$
\begin{eqnarray*}
|\Lambda_n|^{-1} \left( |C_{\Lambda_n}(\varphi;h_k,\ParV)| + |C_{\Lambda_n}(\varphi;h_k,\ParV^\prime)|\right)
&\leq& 2 |\Lambda_n|^{-1}\ism[\Lambda_n \times \mSp]{\max_{\ParV \in \SpPar} |h_k(x^m,\varphi;\ParV)|e^{-V(x^m|\varphi;\ParV)}} \\
&&+ 2 |\Lambda_n|^{-1} \sum_{x^m \in \varphi_{\Lambda_n}} \max_{\ParV\in \SpPar} |h_k(x^m,\varphi\setminus x^m;\ParV)| \\
&\leq & \gamma_1, 
\end{eqnarray*}
where 
$$\gamma_1:= 4\times \max_{k=1,\ldots,K} \left( \Esp\left(
\max_{\ParV\in \SpPar} |f_k(0^M,\Phi;\ParV)|
 \right) + 
\Esp\left(
\max_{\ParV\in \SpPar} |h_k(0^M,\Phi;\ParV)| e^{-V(0^M|\Phi;\ParVT)}
 \right)\right).$$
Therefore for all $n\geq n_0$
$$|U_{\Lambda_n}(\varphi;\Vect{h},\ParV)-U_{\Lambda_n}(\varphi;\Vect{h},\ParV^\prime)| \leq \gamma_1 \times \frac1{|\Lambda_n|}\sum_{k=1}^K \left( A_{\Lambda_n}(\varphi;h_k,\ParV,\ParV^\prime) + B_{\Lambda_n}(\varphi;h_k,\ParV,\ParV^\prime) \right),
$$
where 
\begin{eqnarray*}
A_{\Lambda_n}(\varphi;h_k,\ParV,\ParV^\prime) &:=& \ism[\Lambda_n\times \mSp]{|f_k(x^m,\varphi;\ParV)-f_k(x^m,\varphi;\ParV^\prime)| }\\
B_{\Lambda_n}(\varphi;h_k,\ParV,\ParV^\prime) &=& \sum_{x^m \in \varphi_{\Lambda_n}} |h_k(x^m,\varphi\setminus x^m;\ParV) - h_k(x^m,\varphi\setminus x^m;\ParV^\prime)|.
\end{eqnarray*}
Under Assumption \textbf{[C4]}, one may apply the mean value theorem in ${\RR[p]}$ as follows: there exist $\Vect{\xi}^{(1)},\ldots, \Vect{\xi}^{(p)} \in \left[ \min(\theta_1,\theta_1^\prime),\max(\theta_1,\theta_1^\prime)\right]\times \ldots \times \left[ \min(\theta_p,\theta_p^\prime),\max(\theta_p,\theta_p^\prime)\right]$ such that for all $\varphi\in \cSp$
\begin{eqnarray*}
A_{\Lambda_n}(\varphi;h_k,\ParV,\ParV^\prime) &=&  \ism[\Lambda_n\times \mSp]{
\sum_{j=1}^p (\theta_j - \theta_j^\prime) (\Vect{f_k}^{(1)}(x^m,\varphi;\Vect{\xi}^{(1)}_j))_j \; }
\\
&\leq& p\times \| \ParV -\ParV^\prime\| \ism[\Lambda_n \times \mSp]{\max_{\ParV\in \SpPar} \| \Vect{f_k}^{(1)}(x^m,\varphi;\ParV) \| \; }
\end{eqnarray*}
In a similar way, one may prove that for $P_{\ParVT}-$a.e. $\varphi$
$$
B_{\Lambda_n}(\varphi;h_k,\ParV,\ParV^\prime) \leq p\times \| \ParV -\ParV^\prime\| 
\sum_{x^m \in \varphi_{\Lambda_n}} \max_{\ParV\in \SpPar} \| \Vect{h_k}^{(1)}(x^m,\varphi\setminus x^m;\ParV) \|.
$$
Under Assumption \textbf{[C4]}, there exists $n_1(k)\geq 1$ such that for all $n\geq n_1(k)$, we have for $P_\ParVT-$a.e. $\varphi$
$$
\frac1{|\Lambda_n|} \big( A_{\Lambda_n}(\varphi;h_k,\ParV,\ParV^\prime)+B_{\Lambda_n}(\varphi;h_k,\ParV,\ParV^\prime) \big) \leq \gamma_2 \| \ParV-\ParV^\prime\|
$$
where 
$$\gamma_2:= 2p\times \max_{k=1,\ldots,K} \left(
\Esp\left( 
\max_{\ParV\in \SpPar} \| \Vect{f_k}^{(1)}(0^M,\Phi;\ParV)\|
 \right) +
\Esp\left(
\max_{\ParV\in \SpPar} \| \Vect{h_k}^{(1)}(0^M,\Phi;\ParV)\|
e^{-V(0^M|\Phi;\ParVT)}
\right)
\right).$$
We finally obtain the following upper-bound for $P_{\ParVT}-$a.e. $\varphi$, for all $\ParV,\ParV^\prime$ such that $\|\ParV-\ParV^\prime\|\leq 1/\ell$ and for all $n\geq N=\max(n_0,\max_k n_1(k))$
$$
|U_{\Lambda_n}(\varphi;\Vect{h},\ParV)-U_{\Lambda_n}(\varphi;\Vect{h},\ParV^\prime)| \leq \gamma \times \frac1\ell,
$$
with $\gamma=K\times\gamma_1 \times \gamma_2$ and therefore $W_n(\varphi,1/\ell)\leq \gamma \times \frac1\ell$.
Finally, since
$$
\limsup_{n\to +\infty} \left\{ W_n\left(\varphi, \frac1\ell \right) \geq \frac \gamma\ell \right\} = \bigcap_{m \in \NN}\bigcup_{n \geq m } \left\{ W_{n}\left( \varphi,\frac1\ell \right) \geq \frac \gamma\ell  \right\} \subset
\bigcup_{n \geq N} \left\{ W_{n}\left(\varphi, \frac1\ell \right) \geq \frac \gamma\ell  \right\}
$$
for $P_{\ParVT}-$a.e. $\varphi$, the expected result (\ref{modCont}) is proved.

\noindent\textit{Conclusion step.}
Steps 1 and 2 ensure the fact that we can apply Theorem 3.4.3 of \cite{B-Guy95} which asserts the almost sure convergence for minimum contrast estimators.

\subsection{Proof of Theorem~\ref{prop-norm}}

The proof is based on a classical result concerning asymptotic normality for minimum contrast estimators 
{\it e.g.} Theorem 3.4.5 of \cite{B-Guy95}. We split it into two different steps.\\

\noindent {\it Step 1.} Asymptotic normality of $\Vect{U}^{(1)}_{\Lambda_n}(\Phi;\Vect{h},\ParVT)$.

We start with a Lemma which states a central limit theorem. Its proof uses a conditional centering assumption as in \cite{A-JenKun94}, which holds due to the particular form of the contrast function \eqref{eq-defC}. This scheme of proof is now well-known and we mainly refer to previous works for the technical details.

\begin{lemma} \label{lem-CL}
Under the Assumptions \textbf{[N1]}, \textbf{[N2]} and \textbf{[N3]}, the following convergence holds in distribution, as $n\to +\infty$
\begin{equation}\label{eq-CL}
|\Lambda_n|^{-1/2} \; \widetilde{\Vect{C}}_{\Lambda_n}(\Phi;\Vect{h},\ParVT) \stackrel{d}{\to} \mathcal{N}(0,\Mat{\Sigma}(\Vect{h},\ParVT)),
\end{equation}
where $\Mat{\Sigma}(\Vect{h},\ParVT)$ and  $\widetilde{\Vect{C}}_{\Lambda_n}(\varphi;\Vect{h},\ParVT)$ are defined in Theorem~\ref{prop-norm}.
\end{lemma}

\begin{proof}
The vector $\widetilde{\Vect{C}}_{\Lambda_n}(\varphi;\Vect{h},\ParVT)$ (of length $K$) corresponds to the vector of the $h_k-$residuals for $k=1,\ldots,K$ computed on the same domain $\Lambda_n$ with $\widehat{\ParV}=\ParVT$, see \cite{A-BadTurMolHaz05} for a definition and practical study of this concept of residuals. The asymptotic behavior of the residuals process has been investigated in \cite{A-CoeLav10}. In particular, with the notation of the present paper, the asymptotic normality of the vector $\widetilde{\Vect{C}}_{\Lambda_n}(\Phi;\Vect{h},\widehat{\ParV})$ for general $\widehat{\ParV}$ is obtained (see Proposition~4 in \cite{A-CoeLav10}). When $\widehat{\ParV}=\ParVT$, the assumptions and the asymptotic covariance matrix of Proposition~4 in \cite{A-CoeLav10} respectively reduce to \textbf{[N1-3]} and~(\ref{matC}).
\end{proof}

\bigskip

\noindent Now, according to the definition of $U_{\Lambda_n}(\varphi;\Vect{h},\ParVT)$, we have
$$
\Vect{U}^{(1)}_{\Lambda_n}(\varphi;\Vect{h},\ParVT) = 2|\Lambda_n|^{-2} \; \sum_{k=1}^K \Vect{C}^{(1)}_{\Lambda_n}(\varphi;h_k,\ParVT) C_{\Lambda_n}(\varphi;h_k,\ParVT).
$$
Under Assumption \textbf{[C4]}, one may apply the ergodic Lemma \ref{ergodiclemma}, in order to derive $P_\ParVT-$a.s., as $n\to +\infty$ 
\begin{equation} \label{convC1}
|\Lambda_n|^{-1} \Vect{C}^{(1)}_{\Lambda_n}(\Phi;h_k,\ParVT) \to \Esp\left(
\Vect{f_k}^{(1)}(0^M,\Phi;\ParVT) - \Vect{h_k}^{(1)}(0^M,\Phi;\ParVT)  e^{-V(0^M|\Phi;\ParVT)}
 \right).
\end{equation}
It is easily checked that this expectation reduces  to $- \Vect{\mathcal{E}}(h_k,\ParVT)$, this vector of length $p$ being defined by \\
$\Vect{\mathcal{E}}(h_k,\ParVT):= \Esp \left( 
h_k(0^M,\Phi;\ParVT) \Vect{V}^{(1)}(0^M|\Phi;\ParVT) e^{-V(0^M|\Phi;\ParVT)} 
\right)$. Let us denote by $\Mat{\mathcal{E}}(\Vect{h},\ParVT)$ the $(p,K)$ matrix $\left(\Vect{\mathcal{E}}(h_1,\ParVT),\ldots,\Vect{\mathcal{E}}(h_K,\ParVT)\right)$, then we get the following decomposition
\begin{eqnarray*}
|\Lambda_n|^{1/2} \Vect{U}^{(1)}_{\Lambda_n}(\Phi;\Vect{h},\ParVT) &=& 2 |\Lambda_n|^{1/2} \times |\Lambda_n|^{-2}  \sum_{k=1}^K \Vect{C}^{(1)}_{\Lambda_n}(\Phi;h_k,\ParVT) {C}_{\Lambda_n}(\Phi;h_k,\ParVT) \\
&=& - 2|\Lambda_n|^{-1/2} {\Mat{\mathcal{E}}(\Vect{h},\ParVT)} \widetilde{\Vect{C}}_{\Lambda_n}(\Phi;\Vect{h},\ParVT) \\
&&+ 2 \sum_{k=1}^K \left[ |\Lambda_n|^{-1} \Vect{C}^{(1)}_{\Lambda_n}(\Phi;h_k,\ParVT)- (-\Vect{\mathcal{E}}(h_k,\ParVT)) \right] |\Lambda_n|^{-1/2} C_{\Lambda_n}(\Phi;h_k,\ParVT).
\end{eqnarray*}
According to (\ref{convC1}) and Lemma~\ref{lem-CL}, Slutsky's Theorem implies that for any $k\in \{1,\ldots,K\}$, 
$$\left( |\Lambda_n|^{-1} \Vect{C}^{(1)}_{\Lambda_n}(\Phi;h_k,\ParVT)-\Vect{\mathcal{E}}(h_k,\ParVT) \right) |\Lambda_n|^{-1/2} C_{\Lambda_n}(\Phi;h_k,\ParVT) \stackrel{P}{\to} 0,
$$ 
as $n\to +\infty$, the zero here being a vector of length $p$. Using again Lemma~\ref{lem-CL}, we finally reach the following convergence in distribution as $n\to +\infty$
$$
|\Lambda_n|^{1/2} \Vect{U}^{(1)}_{\Lambda_n}(\Phi;\Vect{h},\ParVT)  \stackrel{d}{\to} \mathcal{N}\left(0, 4 {\Mat{\mathcal{E}}(\Vect{h},\ParVT)} \; \Mat{\Sigma}(\Vect{h},\ParVT) \; \tr{\Mat{\mathcal{E}}(\Vect{h},\ParVT)} \right),
$$
where $\Mat{\Sigma}(\Vect{h},\ParVT)$ is defined by~(\ref{matC}).\\

\noindent {\it Step 2.} Convergence of $\Mat{U}_{\Lambda_n}^{(2)}(\Phi;\Vect{h},\ParV)$ for $\ParV \in \mathcal{V}(\ParVT)$

According to our definition and Assumption \textbf{[N4]}, the $(p,p)$ matrix $\Mat{U}_{\Lambda_n}^{(2)}(\varphi;\Vect{h},\ParV)$ is defined for $i,j=1,\ldots,p$ by
$$
\left( \Mat{U}_{\Lambda_n}^{(2)}(\varphi;\Vect{h},\ParVT) \right)_{ij}= 2 |\Lambda_n|^{-2} \sum_{k=1}^K \left\{ 
\left( \Mat{C}_{\Lambda_n}^{(2)}(\varphi;h_k;\ParV) \right)_{ij} C_{\Lambda_n}(\varphi;h_k,\ParV) +  \left( \Vect{C}^{(1)}_{\Lambda_n}(\varphi;h_k,\ParV) \right)_i \left( \Vect{C}^{(1)}_{\Lambda_n}(\varphi;h_k,\ParV) \right)_j 
\right\}.$$
Note also that $\Vect{C}_{\Lambda_n}^{(1)}(\varphi;h_k,\ParV)$ and $\Mat{C}_{\Lambda_n}^{(2)} (\varphi;h_k,\ParV)$ are defined by 
\begin{eqnarray*}
\Vect{C}_{\Lambda_n}^{(1)}(\varphi;h_k,\ParV) &=& 
\ism[\Lambda_n\times\mSp]{ \Vect{f}^{(1)}_k(x^m,\varphi;\ParV) } \nonumber  -  \sum_{x^m \in \varphi_{\Lambda_n}} \Vect{h}^{(1)}_k(x^m,\varphi\setminus x^m;\ParV) \\
\Mat{C}_{\Lambda_n}^{(2)} (\varphi;h_k,\ParV) &=& \ism[\Lambda_n\times\mSp]{ \Mat{f}^{(2)}_k(x^m,\varphi;\ParV) } \nonumber  -  \sum_{x^m \in \varphi_{\Lambda_n}} \Mat{h}^{(2)}_k(x^m,\varphi\setminus x^m;\ParV).
\end{eqnarray*}

Under Assumption \textbf{[N4]}, then for all $i,j=1,\ldots,p$ and for any $k=1,\ldots,K$, each normalized term  $|\Lambda_n|^{-1} C_{\Lambda_n}(\Phi;h_k,\ParV)$, $|\Lambda_n|^{-1}\Vect{C}_{\Lambda_n}^{(1)}(\Phi;h_k,\ParV)$ and $|\Lambda_n|^{-1}\Mat{C}_{\Lambda_n}^{(2)} (\Phi;h_k,\ParV)$ satisfies the assumptions of the ergodic Lemma \ref{ergodiclemma}. Therefore, for any $\ParV \in \mathcal{V}(\ParVT)$, there exists a matrix $\Mat{U}^{(2)}(\Vect{h},\ParV)$ such that $P_\ParVT-$a.s.
$$
\Mat{U}_{\Lambda_n}^{(2)}(\Phi;\Vect{h},\ParV) \to \Mat{U}^{(2)}(\Vect{h},\ParV).
$$
This justifies that, for $n$ large enough, in a neighborhood of $\ParVT$, $\big(\Mat{U}_{\Lambda_n}^{(2)}(\varphi;\Vect{h},\ParV)\big)_{ij}$ is uniformly bounded by $2\times \max_{\ParV \in \mathcal{V}(\ParVT)} |\big( \Mat{U}^{(2)}(\Vect{h},\ParV)\big)_{ij}|$ for $P_\ParVT-$a.e. $\varphi$. When $\ParV=\ParVT$, recall, from (\ref{GNZ}), that $|\Lambda_n|^{-1}C_{\Lambda_n}(\Phi;h_k,\ParVT)$ converges almost surely to zero and that~(\ref{convC1}) holds. Hence, $\Mat{U}^{(2)}(\Vect{h},\ParVT)$ reduces to $2 \Mat{\mathcal{E}}(\Vect{h},\ParVT) \tr{\Mat{\mathcal{E}}(\Vect{h},\ParVT)}$. \\

\noindent {\it Conclusion Step.} From Theorem 3.4.5 of \cite{B-Guy95}, Steps 1 and 2 ensure that the normalized difference  $|\Lambda_n|^{1/2}\left( \Mat{U}^{(2)}(\Vect{h},\ParVT)\left( \widehat{\ParV}_n(\Phi)-\ParVT \right)  -\Vect{U}_{\Lambda_n}^{(1)}(\Phi;\Vect{h},\ParVT) \right)$ converges in probability to 0, which is the expected result.

%%%%%%%%%%%%%%%%%%%%%%%%%%%%%%%%%%%%%%%%%%%%%%%%%%%%%%%%%%%%%%%%%%%%%%%%%%%%%
%%%%%%%%%%%%%%%%%%%% BIBLIO %%%%%%%%%%%%%%%%%%%%%%%%%%%%%%%%%%%%%%%%%%%%%%%%%
%%%%%%%%%%%%%%%%%%%%%%%%%%%%%%%%%%%%%%%%%%%%%%%%%%%%%%%%%%%%%%%%%%%%%%%%%%%%%

\section*{Acknowledgements}  
We are grateful to the associate editor, the anonymous referee and to Hans Zessin for their valuable comments and advices.

\bibliography{takacs}

\end{document}